\documentclass[enabledeprecatedfontcommands,a4paper]{scrartcl}
\usepackage{amssymb}
\usepackage{amsmath}
\usepackage{amsthm}
\usepackage{latexsym}
\usepackage{enumerate}
\usepackage{graphicx,color}
\usepackage{hyperref}
\usepackage{dsfont}
\usepackage{mathtools}
\usepackage[german,english]{babel}
\usepackage{prodint}
\usepackage[space]{grffile}
\usepackage{xcolor}
\usepackage{subfig}

\numberwithin{equation}{section}

\numberwithin{figure}{section}

\theoremstyle{plain}
\newtheorem{theorem}{Theorem}[section]

\newtheorem{lemma}[theorem]{Lemma}

\theoremstyle{definition}
\newtheorem{definition}[theorem]{Definition}
\newtheorem{assumption}[theorem]{Assumption}

\newtheorem{remark}[theorem]{Remark}

\renewcommand{\d}{\mathrm{d}}
\newcommand{\E}{\mathbb{E}}

\allowdisplaybreaks

\title{On the Estimation of Bivariate Conditional Transition Rates}
\author{Theis Bathke
	\footnote{Carl von Ossietzky Universit\"{a}t Oldenburg, Institut f\"{u}r Mathematik, 26111 Oldenburg, Germany, theis.bathke@uni-oldenburg.de}
}

\date{\today
}

\begin{document}
	
	\maketitle
	
	\begin{abstract}
 Recent literature has found conditional transition rates to be a useful tool for avoiding Markov assumptions in multi-state models. While the estimation of univariate conditional transition rates has been extensively studied, the intertemporal dependencies captured in the bivariate conditional transition rates still require a consistent estimator. We provide an estimator that is suitable for censored data and emphasize the connection to the rich theory of the estimation of bivariate survival functions.
 Bivariate conditional transition rates are necessary for various applications in the survival context but especially in the calculation of moments in life insurance mathematics.
	\end{abstract}
	

	\section{Introduction}
	In survival statistics, the Kaplan-Meier estimator, as seen in Kaplan \& Meier \cite{KaplanMeier1958}, is widely used to obtain non-parametric estimation of survival functions for various modelling purposes when censoring is involved. This applies to various scenarios, such as life insurance modelling, time to machine failure, or the time to illness after the application of medicine. However, this estimator can only take into account a single event and therefore cannot capture dependencies between different events included in the data. Consequently, its suitability for analysing data sets, particularly those arising from genetic epidemiological studies, for disease events of family members, or other similar dependencies, is limited. Many different estimators have been proposed to fill this gap. Dabrowska \cite{Dabrowska1988} uses two-dimensional product integrals to estimate the bivariate survival function, whilst Prentice \& Cai \cite{PrenticeCai1992} employ Volterra integral equation techniques to connect the covariance function and the marginal survivor functions with the  bivariate survivor function. Van der Laan \cite{VanderLaan1996} uses non-parametric maximum likelihood estimation, and recently, Prentice \& Zhao \cite{PrenticeZhao2016} extended the work of Dabrowska with a recursive formula for arbitrary numbers of failure times. Gill \cite{Gill1993} gives a reason why there are so many different estimators: the fact that in two and higher dimensions there are different ways to get from one point in the two-dimensional plane to another point in the two-dimensional plane. \\
	Another important concept in survival analysis is the notion of recurrent events used to study recurring episodes of infection or hospitalisation. This is often modelled by introducing multi-state models, described by Hougaard \cite{Hougaard2000}, which are mostly based on modelling with Markov assumptions. In these models, estimators often target cumulative intensity processes which provide insight into the expected number of upcoming transitions. These estimators are now known as the Nelson-Aalen estimator for transition rates and the Aalen-Johansen estimator of transition probabilities, see Aalen \cite{aalen1978nonparametricn} and Borgan, Gill \& Keiding \cite{borgan2014n}. Datta \& Satten \cite{datta2001validity} discovered that the Nelson-Aalen and Aalen-Johansen estimators remain consistent on data that does not follow the Markov assumption. \textcolor{black}{The use case for multi-state models without the Markov assumption, referred to here as non-Markov models, has been showcased by Putter and Spitoni in \cite{putter2018non}. Their simulation studies demonstrate that, in some instances, Markov estimation introduces a bias when the data does not follow the Markov assumption. In this case, there is a favourable bias-variance trade-off when using the landmark Nelson-Aalen and Aalen-Johansen estimators, which target transition rates and probabilities of a non-Markov model.}\\
	However, transition rates and transition probabilities do not hold full information about the distribution of these state processes without the Markov assumption, as they do not capture the temporal interdependencies between consecutive jump events, necessary, for example, for calculating variances or higher moments. To fill this gap, Bathke \& Christiansen \cite{bathke2022two} have recently proposed a novel framework involving bivariate transition rates but without suitable non-parametric estimators. This paper aims to address the lack of estimation procedures for general bivariate intensity rates. We demonstrate the relationship to the estimation of bivariate survival functions, particularly the works of Dabrowska \cite{Dabrowska1988} and Prentice \& Cai \cite{PrenticeCai1992}.
Our aim is to construct a landmark estimator and demonstrate its consistency through almost sure uniform convergence. This is a widely accepted method for understanding the asymptotic properties of estimators, as demonstrated by Shorack \& Wellner \cite{shorack2009empirical}. 
Bivariate rates have another important application apart from being used for classical survival analysis applications. They are also key for the calculation of life insurance premiums or reserves of insurance contracts with special payment functions in non-Markov models. \textcolor{black}{This is important because life insurance calculation is generally done in Markov or semi-Markov models, which introduce systematic model risk.} Insurance cash flows typically include transition payments from one state to another, as well as sojourn payments for remaining in the corresponding state. Transition payments can be modelled using jump processes, while sojourn payments can be modelled using indicator functions. Guibert \& Planchet \cite{guibert2018non} were able to estimate reserves in a general non-Markov multi-state model for sojourn payments. Christiansen \cite{christiansen2021calculation} extended this to reserves of insurance cash flows with both transition and sojourn payments, using path-independent payment functions. He showed that the consistency properties of the estimators can be transferred to reserve estimation. Our contribution is to show that our bivariate estimators of transition rates and probabilities can be used to estimate reserves for specific path-dependent cash flows and second moments of future liabilities, which are important for many types of risk analysis.\\
	The paper is structured as follows. In Section \ref{Chap:2} the general multi-state modelling framework is introduced, here we connect jump processes and indicator processes. In Section \ref{Chap: Section 3} we introduce our bivariate notation. Section \ref{Chap:3 forward rates} introduces the definition of bivariate conditional rates and the corresponding integral equations, which will be connected to bivariate Peano-series. Section \ref{Chap: section 4} focuses on estimating the bivariate conditional rates using a bivariate landmark Nelson-Aalen estimator, as well as the bivariate transition probabilities using the bivariate landmark Aalen-Johansen estimator. In this Section we introduce Theorem \ref{theo: convergence of Lambda and P} which is the main contribution of this paper. Section \ref{Chap: Section 5} demonstrates how these estimators can be utilised to form estimators for cumulative expected values of general cash flows with interdependencies between two time points in the context of life insurance. Section \ref{Chap: numerical example} introduces a numerical example, where the set up we introduce in Section \ref{Chap: Section 5} is used to form estimators for expected future cash flows in the case of scaled payments in life insurance.
	
	\section{Multi-state modelling framework}\label{Chap:2}
	
	Let $(\Omega,\mathcal{A},\mathbb{P})$ be a probability space with a filtration $\mathbb{F}=(\mathcal{F}_t)_{t \geq 0}$. We consider an individual with the state-process 
		\begin{align*}
		Z=(Z(t))_{t\geq 0},
	\end{align*} which describes the current state that the individual is in. This process is modelled as an adapted c\`{a}dl\`{a}g jump process on a finite state space $\mathcal{Z}$. When modelling these objects, we want to be able to describe past, present, and future developments. Thus, we assume that we are currently at time $s \geq 0$. So the time interval $[0,s]$ represents the past and the present, and the time interval $(s,\infty)$ represents the future. This time point $s$ is arbitrary but fixed throughout this paper.
	Based on $Z$ we define additional processes, first the state indicator processes $(I_i)_{i\in\mathcal{Z}}$ by
	\begin{align*}
		&I_i(t):=\mathds{1}_{\{Z(t)=i\}},\quad t \geq 0,
	\end{align*}
	and second the transition counting processes $(N_{ij})_{i,j\in\mathcal{Z}:i\neq j}$ by
	\begin{align*}
		&N_{ij}(t):=\#\{u\in(0,t]:Z(u^-)=i,Z(u)=j\},\quad t \geq 0.
	\end{align*}
	We generally assume that
	\begin{align}
		\mathbb{E}[N_{ij}(t)^2]<\infty,\quad t\geq 0\;,i,j\in\mathcal{Z},\; i\neq j,\label{eq: finite expected value of N}
	\end{align}
	which in particular implies that $Z$ has almost surely no explosions. 
	Let
	\begin{align} \label{eq: DiagSprung} \begin{split}
			N_{ii}(t)&:=-\sum_{j\in \mathcal{Z}}(N_{ij}(t)-N_{ij}(s)),\quad t>s\quad i\in\mathcal{Z}. 
		\end{split}
	\end{align}
	This construction is done so that we can use the abbreviation
	\begin{align*}
		\d N_{ii}(t)=-\sum_{j\in\mathcal{Z}}\d N_{ij}(t), \quad t>s,\quad i\in\mathcal{Z}.
	\end{align*} 
	The definition \eqref{eq: DiagSprung} and many other following definitions are only defined for $t>s$. For retrospective modelling, this restriction can be removed if necessary, see \cite{bathke2022two} for these extensions of the basic theory. Since time $s$ is fixed, we omit it in the notation, but one should keep in mind the dependence of many of our definitions on the parameter $s$.
	The following equation shows a useful direct link between the processes $(N_{ij})_{i,j\in\mathcal{Z}}$ and $(I_i)_{i\in\mathcal{Z}}$:
	\begin{align}
		I_i(t)=I_i(s)+\sum_{j\in \mathcal{Z}} \;\int\displaylimits_{(s,t]} N_{ji}(\d u),\quad t\geq s, i\in\mathcal{Z}\label{I_represent_by_N}.
	\end{align} 
	The latter integral and all following integrals in this paper are meant as path-wise Lebesgue-Stieltjes integrals, which in the case of \eqref{I_represent_by_N} only exists almost surely.
	
	The sigma algebra $\mathcal{F}_s$ represents the information available at time $s$. In practice, we often use a reduced information set $\mathcal{G}_s=\sigma(\xi)$ for evaluations, which is generated by a discrete random variable $\xi$. For general two-dimensional computations we assume
	\begin{align}\label{mathcalGs}
		\sigma(Z(s))\subseteq \mathcal{G}_s\subseteq \mathcal{F}_s.
	\end{align} This assumption is needed for calculating bivariate conditional transition rates and bivariate transition probabilities in the general non-Markov model, see \cite{bathke2022two}.\\
	The special case $\mathcal{G}_s= \sigma(Z(s))$ is known as the as-if Markov model, since we use only data that would have been used in the Markov model as well.
	The choice of $\mathcal{G}_s$ can be influenced by many factors. Some of these are listed below: \begin{itemize}
		\item Numerical complexity,
		\item Lack of data,
		\item Regulatory requirements.
	\end{itemize}
 Currently, Markov modelling is widely used in practice. Therefore, most of the available data is of Markov type. Thus, the natural extension would be the as-if Markov model, as it combines classical Markov information with a non-Markov model. In addition, the numerical complexity increases whenever we work with a broader information model. \textcolor{black}{ We therefore focus on an estimation procedure that conditions on discrete random variables, which includes the as-if Markov model because $\#\mathcal{Z}<\infty$.}

	\section{Notation}\label{Chap: Section 3}
	For two-dimensional variables, definitions of processes and integrals, we use $\boldsymbol t=(t_1,t_2)$, $f(\boldsymbol u^-)=f(u_1^-,u_2^-)$ for a bivariate upper continuous function $f$ that has uniform lower limits, and $f(\d \boldsymbol u)=f(\d u_1,\d u_2)$ for a two-variable function $f$ with finite two-dimensional total variation, see the Appendix \ref{def:variation} for the definition of the total variation used in this paper. 
	In addition, we use partial ordering on $\mathbb{R}^2$, meaning that $\boldsymbol x\leq \boldsymbol y$ if and only if $x_1\leq y_1$ and $x_2\leq y_2$. We extend this notation so that $\boldsymbol x$ can be set into relation with a one-dimensional constant. For example, $\boldsymbol x\leq y$ iff $x_1\leq y$ and $x_2\leq y$. Consequently, $(s,\boldsymbol t]$ is a two-dimensional rectangle that is open at the bottom and the left boundaries and closed at the top and right boundaries with the corner points $(s,s)$ and $(t_1,t_2)$.  Apart from two-dimensional variables, we use various indices to model different jumps at different times. For this we introduce the two-dimensional indices $\boldsymbol i=(i_1,i_2)$. 
	\section{ Conditional Transition Rates}\label{Chap:3 forward rates}
	
	This section introduces the conditional transition rates of Christiansen \& Furrer \cite{christiansen2021calculation} and Bathke \& Christiansen \cite{bathke2022two}. We assume that we are currently at time $s$ and have the information $\mathcal{G}_s=\sigma(\xi)$ available, where $\xi$ is a discrete random variable called the landmark. Since the parameter $s$ is fixed, we omit it in the notation
	
	Let $P_{z,i}=(P_{z,i}(t))_{t>s}$ and $Q_{z,ij}=(Q_{z,ij}(t))_{t>s}$ for $ i,j \in \mathcal{Z}$ be the almost surely unique c\`{a}dl\`{a}g paths that satisfy
	\begin{align*}
		P_{z,i}(t)&=\E[ I_i(t)|\xi=z],\\
		Q_{z,ij}(t) &= \E[ N_{ij}(t)|\xi=z].
	\end{align*}
	
	Let $\boldsymbol P_{z,\boldsymbol i}= (\boldsymbol P_{z,\boldsymbol i}(\boldsymbol t))_{t_1,t_2 >s}$ and $\boldsymbol Q_{z,\boldsymbol i\boldsymbol j}=(\boldsymbol Q_{z,\boldsymbol i \boldsymbol j}(\boldsymbol t))_{t_1,t_2 >s}$ for $\boldsymbol i,\boldsymbol j \in\mathcal{Z}^2$ be the almost surely unique surfaces that are c\`{a}dl\`{a}g in each variable and satisfy
	\begin{align*}
		\boldsymbol P_{z,\boldsymbol i}(\boldsymbol t)&=\mathbb{E}[I_{i_1}(t_1)I_{i_2}(t_2)|\xi=z],\\
		\boldsymbol Q_{z,\boldsymbol i \boldsymbol j}(\boldsymbol t)&=\mathbb{E}[ N_{i_1j_1}(t_1)N_{i_2j_2}(t_2)|\xi=z],
	\end{align*} for all these definitions refer to a regular conditional probability $\mathbb{P}(\cdot|\xi=z)$.	These functions are almost surely unique because the c\`{a}dl\`{a}g property ensures a almost surely unique definition by their values at rational times, which are countably many.
\textcolor{black}{
\begin{remark}
	Following section \ref{Chap:2}, $\xi=Z(s)$ is a natural example of a landmark, as seen for example in the as-if Markov model. In this case, $z\in\mathcal{Z}$. But the landmark can also include discrete external landmarks such as gender. So, in general, $z$ need not to be in $\mathcal{Z}$. The fact that the landmark $\xi$ is discrete is important for the landmarking ideas we use in this paper. See \cite{bladt2023conditional} for an approach to conditioning on continuous information.
\end{remark}}
	\begin{definition}\label{DefLambda}
		For $ i,j \in \mathcal{Z}$ and $t\in(s,\infty)$ let the function $(\Lambda_{z,ij}(t))_{t \geq s}$ be defined by
		\begin{align}\label{DefEqLambda}
			\Lambda_{z,ij}(t) = \;\int\displaylimits_{(s,t]} \frac{\mathds{1}_{\{{P}_{z,i}(u^-)>0\}}}{{P}_{z,i}(u^-)} Q_{z,ij}(\d u),
		\end{align}
		and for $\boldsymbol i,\boldsymbol j \in \mathcal{Z}^2$ and $\boldsymbol t\in(s,\infty)^2$ let the function  $(\boldsymbol\Lambda_{z,\boldsymbol i \boldsymbol j}(\boldsymbol t))_{\boldsymbol t \geq s}$
		be defined by
		\begin{align}\label{DGL of g and Gamma}
			\boldsymbol \Lambda_{z,\boldsymbol i \boldsymbol j}(\boldsymbol t) = \;\int\displaylimits_{(s,\boldsymbol t]} \frac{\mathds{1}_{\{{\boldsymbol P}_{z,\boldsymbol i}(\boldsymbol u^-)> 0\}}}{{\boldsymbol P}_{z,\boldsymbol i}(\boldsymbol u^-)}\boldsymbol Q_{z,\boldsymbol i \boldsymbol j}(\d \boldsymbol u),
		\end{align}where we use the convention $\frac{0}{0}:=0$.
	\end{definition}
	In case of $t>s$ and $\boldsymbol t> s$ we denote $ \Lambda_{z,ij}(\d t )$ as (univariate) conditional transition rate and $\Lambda_{z,\boldsymbol i \boldsymbol j}(\d \boldsymbol t)$ as bivariate conditional transition rate. In order to ensure existence of \eqref{DefEqLambda} and \eqref{DGL of g and Gamma} we generally assume that
	\begin{align}\label{IntegrabilCond}\begin{split}
			&\;\Lambda_{z,ij}(t) < \infty, \quad t \geq s,\\
			&\;\boldsymbol \Lambda_{z,\boldsymbol i \boldsymbol j}(\boldsymbol t)<\infty,\quad t \geq s,
	\end{split}\end{align}
	almost surely for all $i,j\in\mathcal{Z},\boldsymbol i,\boldsymbol j \in\mathcal{Z}$.
	According to Christiansen \cite{christiansen2021calculation} it holds that
	\begin{align}\label{eq: GenKolmForwardEquation}
		P_{z,i}(t)&=P_{z,i}(s)+\sum_{j\in\mathcal{Z}}\;\int\displaylimits_{(s,t]} {P}_{z,j}(u^-){\Lambda}_{z,ji}(\d u), \quad t \geq s
	\end{align}
	for all $i \in \mathcal{Z}$. This is equivalent to the product integral\begin{align*}
		P_z(t)=P_z(s)\Prodi_{(s,t]}(Id+\Lambda_z(\d u)),\quad t>s.
		\end{align*} 
	This is a generalisation of Kolmogorov's forward equation to non-Markov models.
	Bathke \& Christiansen \cite{bathke2022two} extended this result to bivariate conditional transition rates with the integral equation
		\begin{align}\label{2dimKolmForwEq}\begin{split}
				\boldsymbol P_{z,\boldsymbol i}(\boldsymbol t)&=\boldsymbol P_{z,\boldsymbol i}(\boldsymbol s)
				+P_{z,i_2}(s)\left(P_{z,i_1}(t_1)-P_{z,i_1}(s)\right) +P_{z,i_1}(s)\left(P_{z,i_2}(t_2)-P_{z,i_2}(s)\right)\\
				&\quad +\;\int\displaylimits_{(s,\boldsymbol t]}\sum_{\substack{\boldsymbol j\in\mathcal{Z}^2}}{\boldsymbol P}_{z,\boldsymbol j}(\boldsymbol u^-) {\boldsymbol\Lambda}_{z,\boldsymbol j \boldsymbol i}(\d \boldsymbol u),
		\end{split}\end{align}
		for $\boldsymbol t \geq s$ and $\boldsymbol i=(i_1,i_2) \in \mathcal{Z}^2$.
		
		 These integral equations are similar to the inhomogeneous Volterra equations found in the survival setting, see Prentice and Cai \cite{PrenticeCai1992}.\\
		It should be noted here that these processes depend on the current time $s$, which is omitted only for the sake of conciseness.
	\section{Bivariate Conditional Nelson-Aalen and Aalen-Johansen Estimation }\label{Chap: section 4}
	\textcolor{black}{We begin this section with the general non-parametric estimation setup, which is similar to the non-Markov setup in Christiansen \cite{christiansen2021calculation}}. \textcolor{black}{ The non-parametric approach of this paper is deeply rooted in the fact that many practitioners, for example in the German actuarial community, use purely non-parametric estimators. Moreover, the non-parametric approach can be used as a preview to better calibrate parametric models. First, we present the general estimation setup.} \\
	   For the rest of this paper we assume that we are observing $n\in\mathbb{N}$ individuals:\begin{align*}
		(Z^1(t))_{t\in[0,R^1]},\ldots,(Z^n(t))_{t\in[0,R^n]},
	\end{align*}
	where $R^m$, $m\in\mathbb{N}_{\leq n}$ is the right censoring time of the data. The additional process $N_{ij}^{m}$ and the random variable $\xi^m$ define the counting process and the information of the $m$-th observed individual. For the following estimators and the convergence results of these estimators, we make some key assumptions about our data. \begin{assumption}\label{as:statistical assumptions}\phantom{ }
		\begin{itemize}
			\item[a)] All observations $(Z^i,R^i)_{i\leq n}$ are independent and identically distributed.
			\item[b)] $R^1$ is stochastically independent of $(Z^1(t))_{t\geq s}$ and $\xi^1$ and almost surely greater than $s$.
		\end{itemize}
	\end{assumption}
It means that we assume that $\xi^m$ is not affected by censoring. That is, it is observable in the data at time $s$ for all individuals.  \textcolor{black}{
\begin{remark}
The estimation model could be defined exactly as Christiansen does in \cite{christiansen2021calculation}, but we think that left censoring and weaker conditions on right censoring are not worth the notational overhead. One main application of censoring or truncation in the actuarial context in section \ref{Chap: Section 5} is that the portfolio of insured, and thus the data we use for estimation, is not fully closed, so we have data that may be right censored but is rarely left censored.
\end{remark}}
\textcolor{black}{\begin{remark} 
		There are many other approaches to the modelling of censoring than the one we present in this paper. There is the example of inverse probability of censoring weighting methods, which allow for additional dependency between the jump process and the censoring mechanism. This in turn requires a parametric censoring model, which distinguishes their approach from our fully non-parametric one, as already mentioned by Bladt and Furrer in \cite{bladt2023conditional}. In addition, current work on IPCW methods in multi-state models uses a simpler setup, see the work of Mostajabi and Datta \cite{mostajabi2013nonparametric} and Siriwardhana, Kulasekera and Datta \cite{siriwardhana2018flexible}.
\end{remark}}
	We continue with a short introduction to the univariate landmark Nelson-Aalen and Aalen-Johansen estimators, see Aalen \cite{aalen1978nonparametricn}, Christiansen \cite{christiansen2021calculation}, and Bladt \& Furrer \cite{bladt2023conditional} for reference.\\
 Let $R$ be a random variable that describes the right censoring and $\xi$ the random variable that generates the information $\mathcal{G}_s$. This allows us to define the expected value of the observable censored jump processes and indicator processes.\begin{align*}
		P_{z,j}^c(t)&:=\mathbb{E}[\mathds{1}_{\{Z_t=j\}}\mathds{1}_{\{t<R\}}\mathds{1}_{\{\xi=z\}}],\\
	Q_{z,jk}^c(t)&:=\mathbb{E}[N_{jk}(t\wedge R)\mathds{1}_{\{\xi=z\}}],\end{align*} for $j,k\in\mathcal{Z}$ and $t>s$. This in turn allows us to define censored  transition rates as
	\begin{align*}
		\Lambda_{z,jk}^{c}(t)&:=\int\displaylimits_{(s,t]}\frac{\mathds{1}_{\{P_{z,j}^c(u^-)>0\}}}{P_{z,j}^c(u^-)}Q_{z,jk}^c(\d u),
	\end{align*}for $j,k\in\mathcal{Z}$ and $t>s$. \\
All of these processes depend on the index $z$ which takes values in the image of $\xi$. For this approach to give comprehensible information, the random variable $\xi$ must only take a finite number of values. This is the reason why this approach only works for a finite information space.\\
One of the reasons for using transition rates in Markov and non-Markov modelling is the following equality. If we take a time $\tau$ such that $\mathbb{P}(\tau\leq R)>0$, we get $$\Lambda_{z,jk}(t)=\Lambda_{z,jk}^c(t),$$ for $j,k\in\mathcal{Z}$ and $t\in(s,\tau]$, see for example the ideas used by Glidden \cite{glidden2002robust} or Christiansen \cite{christiansen2021calculation} in the proof of Theorem 7.3. This means that the censoring does not affect the conditional transition rates as long as $t<\tau$.\\
Estimating processes such as $\Lambda_{z,jk}(t)$ is generally challenging, since one might divide by zero. Bladt \& Furrer \cite{bladt2023conditional} use what they call a perturbation of $\Lambda_{z,jk}(t)$. Applying their idea onto the univariate conditional transition rates leaves us with\begin{align*}
	\Lambda_{z,jk}^{(\epsilon)}(t)&:=\int\displaylimits_{(s,t]}\frac{1}{P_{z,j}^c(u^-)\vee \epsilon}Q_{z,jk}^c(\d u),\\
	P_z^{(\epsilon)}(t)&:=P_z(s)\Prodi_{(s,t]}(Id+\Lambda_z^{(\epsilon)}(\d u)),
\end{align*}for $j,k\in\mathcal{Z}$, $t>s$ and $\epsilon>0$. The parameter $\epsilon$ will be used as the perturbation factor for the rest of this paper. The next step is to define the corresponding estimators for the univariate conditional transition rates and probabilities.
 \begin{align*}	I_{z,j}^{(n)}(t)&:=\frac{1}{n}\sum_{m=1}^n\mathds{1}_{\{\xi^m=z\}}\mathds{1}_{\{t\leq R^m\}}\mathds{1}_{\{Z_t^m=j\}},\\
	{N}^{(n)}_{z,jk}(t)&:=\frac{1}{n}\sum_{m=1}^n\mathds{1}_{\{\xi^m=z\}}N_{jk}^m(t\wedge R^m)\\
	\Lambda_{z,jk}^{(n,\epsilon)}(t)&:=\int\displaylimits_{(s,t]}\frac{1}{I_{z,j}^{(n)}(u^-)\vee \epsilon}N_{z,jk}^{(n)}(\d u),\\
	P_z^{(n,\epsilon)}(t)&:=P_z^{(n)}(s)\Prodi_{(s,t]}(Id+\Lambda_z^{(n,\epsilon)}(\d u)),
	\end{align*}
where $P_{z}^{(n)}(s):=\left(\frac{I^{(n)}_{z,j}(s)}{\sum_{i}I^{(n)}_{z,i}(s)}\right)_{j\in\mathcal{Z}}$, $j,k\in\mathcal{Z}$, and $\epsilon>0$.\\
 \noindent
The next step is to use these ideas for the estimation of bivariate conditional transition rates. \begin{definition}
 	For $\boldsymbol{t}=(t_1,t_2)\in(s,\infty)^2$, and $\boldsymbol i=(i_1,i_2),\boldsymbol j=(j_1,j_2)\in\mathcal{Z}^2$ we want to estimate expected values of the censored jump processes and indicator processes:\begin{align*}
 		\boldsymbol P_{z,\boldsymbol i}^c(\boldsymbol t)&:=\mathbb{E}[I_{i_1}(t_1)I_{i_2}(t_2)\mathds{1}_{\{\boldsymbol t<R\}}\mathds{1}_{\{\xi=z\}}],\\ 		\boldsymbol Q_{z,\boldsymbol{ij}}^c(\boldsymbol t)&:=\mathbb{E}[N_{i_1j_1}(t_1\wedge R)N_{i_2j_2}(t_2\wedge R)\mathds{1}_{\{\xi=z\}}],\end{align*}
 	and the corresponding bivariate conditional transition rates:
 	\begin{align*}
 		\boldsymbol \Lambda_{z,\boldsymbol{ij}}^{c}(\boldsymbol t)&:=\int\displaylimits_{(s,\boldsymbol t]}\frac{\mathds{1}_{\{\boldsymbol P^c_{z,\boldsymbol i}(\boldsymbol u^-)>0\}}}{\boldsymbol P^{c}_{z,\boldsymbol{i}}(\boldsymbol u^-)}\boldsymbol Q_{z,\boldsymbol{ij}}^c(\d \boldsymbol u).
 	\end{align*} 
 Again, $R$ is a random variable which describes right censoring for the state process.
 \end{definition}
One of the reasons for using conditional transition rates in the univariate setting is the fact that they are effective even on censored data. This is one of the main reasons for using transition rates in both Markov and non-Markov models. The following lemma captures this feature in the bivariate situation.
\begin{lemma}\label{lem:change between censored and non-censored data} Bivariate censored conditional transition rates satisfy the equation $$\boldsymbol \Lambda_{z,\boldsymbol{ij}}(\boldsymbol t)=\boldsymbol \Lambda_{z,\boldsymbol{ij}}^c(\boldsymbol t),$$ for $\boldsymbol t \in(s,T]^2$ and $\boldsymbol i,\boldsymbol j\in\mathcal{Z}^2$, if $\mathbb{P}(\boldsymbol t\leq R)>0$.
\end{lemma}
\begin{proof}
	For $\boldsymbol i,\boldsymbol j\in \mathcal{Z}^2$, we use Campbell's theorem \eqref{eq:campbell}, the law of iterated expectation, and Assumption \ref{as:statistical assumptions}: \begin{align*}
	\boldsymbol Q_{z,\boldsymbol{ij}}^c(\d \boldsymbol u)&=\mathbb{E}[\mathds{1}_{\{\xi=z\}}\mathds{1}_{\{\boldsymbol u\leq R\}}N_{i_1j_1}(\d u_1)N_{i_2j_2}(\d u_2)]\\
	&=\mathbb{E}[\mathds{1}_{\{\xi=z\}}\mathbb{E}[\mathds{1}_{\{\boldsymbol u\leq R\}}|(Z_t)_{t\geq s}]N_{i_1j_1}(\d u_1)N_{i_2j_2}(\d u_2)]\\
	&=\mathbb{P}(\xi=z)\mathbb{E}[\mathds{1}_{\{\boldsymbol u\leq R\}}]\mathbb{E}[N_{i_1j_1}(\d u_1)N_{i_2j_2}(\d u_2)|\xi=z]\\
	&=\mathbb{P}(\xi=z)\mathbb{P}(\boldsymbol u\leq R)\boldsymbol Q_{z,\boldsymbol{ij}}(\d \boldsymbol u),
	\end{align*} and with similar calculations $$\boldsymbol P_{z,\boldsymbol i}^c(\boldsymbol u^-)=\boldsymbol P_{z,\boldsymbol i}(\boldsymbol u^-)\mathbb{P}(\xi=z)\mathbb{P}(\boldsymbol u\leq R).$$ 
	 If $\mathbb{P}(\boldsymbol u\leq R)>0$ and $\boldsymbol u> s$, we can plug these equations into the definition of $\boldsymbol \Lambda^c_{\boldsymbol{ij}}(\boldsymbol t)$ and  the factor $\mathbb{P}(\xi=z)\mathbb{P}(\boldsymbol u\leq R)$ cancels out. In the case that $\mathbb{P}(\xi=z)=0$ this statement is still true, because then $\boldsymbol\Lambda_{\boldsymbol{ij}}^c(\boldsymbol t)$ is equal to $0$ due to the previous calculations and $\boldsymbol\Lambda_{\boldsymbol{ij}}(\boldsymbol t)$ is equal to $0$ due to the convention $\frac{0}{0}=0$. Thus, we get the desired equation.
\end{proof}
We can conclude that for the estimation of $\boldsymbol \Lambda_{z,\boldsymbol{ij}}(\boldsymbol t)$, an estimator for $\boldsymbol \Lambda^c_{z,\boldsymbol{ij}}(\boldsymbol t)$ works well, if $\boldsymbol{t}<\boldsymbol \tau$ where we define the endpoint $\boldsymbol \tau$ to be a time-point that satisfies $$\mathbb{P}(\boldsymbol \tau\leq R)>0.$$ In the rest of this paper, we always assume that $\boldsymbol\tau$ satisfies this condition and is greater than $s$. \textcolor{black}{ In practice, $\tau$ can be understood as the right endpoint of the data window, we are able to observe and is mostly predetermined by the application. }
We have difficulties estimating a process like $\boldsymbol \Lambda^c_{z,\boldsymbol{ij}}(\boldsymbol t)$ in the case that $\boldsymbol P^{c}_{z,\boldsymbol{i}}(\boldsymbol u^-)$ is close to $0$. There are different strategies to circumvent this problem. See for example IV.1.2 from Andersen, Borgan and Gill \cite{andersen2012statistical} where they use assumptions on the behaviour of jump processes to circumvent this problem. For the estimation of bivariate survival functions this gets solved by the fact that $ P^{c}( u-)$ goes to zero monotonically as $u$ goes to infinity, which does not generally happen in a multi-state model. We use the perturbation idea of Bladt \& Furrer \cite{bladt2023conditional}, which we adapt from the univariate case. This corresponds to the new process $\boldsymbol\Lambda_{z,\boldsymbol{ij}}^{(\epsilon)}(\boldsymbol t)$ with \begin{align*}
	\boldsymbol\Lambda_{z,\boldsymbol{ij}}^{(\epsilon)}(\boldsymbol t)&:=\int\displaylimits_{( s,\boldsymbol t]}\frac{1}{\boldsymbol P^{c}_{z,\boldsymbol{i}}(\boldsymbol u-)\vee \epsilon}\boldsymbol Q_{z,\boldsymbol{ij}}^c(\d \boldsymbol u),
\end{align*}for $\epsilon>0$, $\boldsymbol i,\boldsymbol j\in \mathcal{Z}^2$ and $\boldsymbol t>s$. This $\boldsymbol\Lambda_{z,\boldsymbol{ij}}^{(\epsilon)}(\boldsymbol t)$ is estimated as a replacement for $\boldsymbol\Lambda_{z,\boldsymbol{ij}}(\boldsymbol t)$.\\
In practice, the use of $\boldsymbol\Lambda^{(\epsilon)}_{z,\boldsymbol{ij}}$ is hardly a problem for two reasons. First, in the proof of Theorem \ref{theo: convergence of Lambda and P}, we see that \begin{align*}
	\boldsymbol\Lambda_{z,\boldsymbol{ij}}(\boldsymbol t)=\lim_{\epsilon\rightarrow\infty}\boldsymbol\Lambda_{z,\boldsymbol{ij}}^{(\epsilon)}(\boldsymbol t),
\end{align*}which means that we get arbitrarily close to the unperturbed conditional transition rates. Second, for practical calculations and estimates, we generally use discretisation of our conditional transition rates and probabilities, i.e. a discretisation in the time and value axis. This means that the only way to get infinitely close to zero is to make a jump to zero. These jumps to zero are not a problem if $\epsilon$ is taken to be smaller than the smallest jump. See Bladt \& Furrer \cite{bladt2023conditional} example after Definition 2.4 for similar reasoning in the case of perturbation of the univariate conditional transition rates. Furthermore, in practice a complete convergence of transition probabilities to zero can be stopped anyway by a forced transition of states. Examples of this are transitions to annuity at certain ages. In implementation, this leads to forced transitions, for example, into the state "dead" or "retired". These forced transitions prevent the different transition probabilities from converging to zero. So in practice we usually get $$\boldsymbol\Lambda_{z,\boldsymbol{ij}}(\boldsymbol t)=\boldsymbol\Lambda_{z, \boldsymbol{ij}}^{(\epsilon)}(\boldsymbol t)$$ for $t\in(s,\boldsymbol \tau]$, $\boldsymbol i,\boldsymbol j\in\mathcal{Z}^2$, and a fixed $\epsilon>0$.
\noindent
We start by defining the estimators for the expected value of censored jump processes and indicator processes.

	\begin{align*}
		\boldsymbol I_{z,\boldsymbol i}^{(n)}(\boldsymbol t)&:=\frac{1}{n}\sum_{m=1}^n\mathds{1}_{\{\xi^m=z\}}\mathds{1}_{\{\boldsymbol t<R^m\}}\mathds{1}_{\{Z_{t_1}^m=i_1\}}\mathds{1}_{\{Z_{t_2}^m=i_2\}},\\ 
			\boldsymbol N_{z,\boldsymbol{ij}}^{(n)}(\boldsymbol t)&:=\frac{1}{n}\sum_{m=1}^n\mathds{1}_{\{\xi^m=z\}}N_{i_1j_1}^m(t_1\wedge R^m)N_{i_2j_2}^m(t_2\wedge R^m),\\ \end{align*} for $\boldsymbol i,\boldsymbol j\in\mathcal{Z}^2$ and $\boldsymbol t\in[s,\infty)$.
	These can be used to define the bivariate landmark Nelson Aalen estimator \begin{align*}
		\boldsymbol\Lambda_{z,\boldsymbol{ij}}^{(n,\epsilon)}(\boldsymbol t)&:=\int\displaylimits_{( s,\boldsymbol t]}\frac{1}{\boldsymbol I^{(n)}_{z,\boldsymbol i}(u^-)\vee \epsilon}\boldsymbol N^{(n)}_{z,\boldsymbol{ij}}(\d u),\end{align*}
	and bivariate landmark Aalen Johansen estimators as $P_{z,ik}^{(n,\epsilon)}$, the unique solution to the integral equation
	\begin{align}\begin{split}
		\boldsymbol P_{z,\boldsymbol i}^{(n,\epsilon)}(\boldsymbol t)&=\boldsymbol P_{z,\boldsymbol i}(s)
		+P_{z,i_2}(s)\left(P_{z,i_1}^{(n,\epsilon)}(t_1)-P_{z,i_1}(s)\right)+P_{z,i_1}(s)\left(P_{z,i_2}^{(n,\epsilon)}(t_2)-P_{z,i_2}(s)\right)\\
		 &\quad +\int\displaylimits_{(s,\boldsymbol t]}\sum_{\substack{ \boldsymbol j\in\mathcal{Z}^2}}{\boldsymbol P}^{(n,\epsilon)}_{z,\boldsymbol j}(\boldsymbol u^-) {\boldsymbol\Lambda}^{(n,\epsilon)}_{z,\boldsymbol{ji}}(\d \boldsymbol u),\label{eq: zwei-dimensionale Kolmogorov Gleichung}
		 \end{split}
	\end{align}for $\boldsymbol i,\boldsymbol j\in\mathcal{Z}^2$ and $\boldsymbol t>s$, where the uniqueness comes from Theorem 4.1 from Bathke \& Christiansen \cite{bathke2022two}.

\textcolor{black}{
	\begin{remark}
		In the construction of the estimation model, we assume that we observe individuals until they are censored. 
		The construction of the landmark Nelson-Aalen estimator does not depend on this knowledge. We only need to observe the individuals until they reach an absorption time, i.e. they jump into an absorption state. This is the case because there can be no jumps after the absorption time, so $\boldsymbol\Lambda_{z,\boldsymbol{ij}}^{(n,\epsilon)}(\boldsymbol t)$ does not depend on this information. And since the bivariate landmark Aalen-Johansen estimator only depends on the landmark Nelson-Aalen estimator, this information is not really needed.
\end{remark}}

To achieve convergence of these estimators, it is necessary to ensure convergence of all the univariate and bivariate estimators. The process begins with the almost certain uniform convergence in $n$ of the univariate processes.

Referring to Bladt \& Furrer \cite{bladt2023conditional}, we get almost sure uniform convergence of the estimators for the univariate conditional transition rates. If $s$ and $\tau$ are selected, so that $\tau>s$, $\mathbb{P}( \tau\leq R)>0$, $ i, j\in\mathcal{Z}$ and Assumptions \ref{as:statistical assumptions} are met, we get
\begin{align}
	\begin{split}\label{eq: Konvergenzen in 1-D}
			&\sup_{t\in(s,\tau]}\left|I_{z,j}^{(n)}(t)-P_{z,j}^c(t)\right|\xrightarrow[n\rightarrow\infty]{a.s.}0\\
	&\sup_{t\in(s,\tau]}\left|N_{z,jk}^{(n)}(t)-Q_{z,jk}^c(t)\right|\xrightarrow[n\rightarrow\infty]{a.s.}0,\\
	&\sup_{t\in(s,\tau]}\left|P_{z,j}^{(n,\epsilon)}(t)-P^{(\epsilon)}_{z,j}(t)\right|\xrightarrow[n\rightarrow\infty]{a.s.}0,\\
	&\sup_{t\in(s,\tau]}\left|\Lambda_{z,jk}^{(n,\epsilon)}(t)-\Lambda_{z,jk}^{(\epsilon)}(t)\right|\xrightarrow[n\rightarrow\infty]{a.s.}0.
	\end{split}
\end{align} The idea of the estimation procedure is to estimate censored conditional transition rates and censored transition probabilities, use these to estimate conditional transition rates because they do not depend on uncensored data, and finally to use these uncensored conditional transition rates to estimate the uncensored transition probabilities via the product integral.
 \begin{theorem}\label{theo:convergence of Lambda and P 1D}
	If the Assumption \ref{as:statistical assumptions} holds, we get uniform convergence of the estimators. For $i,j\in\mathcal{Z}$
	\begin{align*}&\lim_{\epsilon\rightarrow 0}\lim_{n\rightarrow\infty}\sup_{t\in(s,\tau]}\left|\Lambda_{z,jk}^{(n,\epsilon)}(t)-\Lambda_{z,jk}(t)\right|\stackrel{a.s}{=}0,\\
		&\lim_{\epsilon\rightarrow 0}\lim_{n\rightarrow\infty}\sup_{t\in(s,\tau]}\left|P_{z,j}^{(n,\epsilon)}(t)-P_{z,j}(t)\right|\stackrel{a.s}{=}0.
	\end{align*}
\end{theorem}
The proof for this theorem follows at the end of this section. 
\\
The setup we use for the estimation of  bivariate conditional transition rates is based on the same idea.
The aim of the following statements is to assert uniform convergence for the bivariate estimators.
\begin{lemma}\label{theo: Konvegenz der erwarteten Sprünge und der Übergangs wahrscheinlichkeiten}
	If $\boldsymbol \tau\in[s,\infty)$ so that $\mathbb{P}(\boldsymbol \tau\leq R)>0$ and  Assumption \ref{as:statistical assumptions} holds, we get uniform convergence of the estimators $\boldsymbol N_{z,ijkl}^{(n)}$ and $ \boldsymbol I_{z,ij}^{(n)}$:
	\begin{align*}
		&\sup_{\boldsymbol t\in(s,\boldsymbol \tau]}\Big| \boldsymbol N_{z,\boldsymbol{ij}}^{(n)}(\boldsymbol t)-\boldsymbol Q_{z,\boldsymbol{ij}}^c(\boldsymbol t)\Big|\xrightarrow[n\rightarrow\infty]{a.s.} 0,\\
		&\sup_{\boldsymbol t\in(s,\boldsymbol \tau]}\Big| \boldsymbol I_{z,\boldsymbol i}^{(n)}(\boldsymbol t)-\boldsymbol P_{z,\boldsymbol i}^{c}(\boldsymbol t)\Big|\xrightarrow[n\rightarrow\infty]{a.s.} 0,
	\end{align*}for $\boldsymbol i,\boldsymbol j\in\mathcal{Z}^2$.
\end{lemma}
The proof for this lemma follows at the end of this section.\\
These convergences can be transferred to the convergence of conditional transition rates and probabilities. Again, the convergence in $\epsilon$ is taken care of.
\begin{theorem}\label{theo: convergence of Lambda and P}If the Assumption \ref{as:statistical assumptions} holds, we get uniform convergence of the estimators. For $\boldsymbol i,\boldsymbol j\in\mathcal{Z}^2$ 
	\begin{align*}&\lim\limits_{\epsilon\rightarrow 0}\lim\limits_{n\rightarrow\infty}\sup_{\boldsymbol t\in(s,\boldsymbol\tau]}\Big|\boldsymbol\Lambda_{z,\boldsymbol{ij}}^{(n,\epsilon)}(\boldsymbol t)-\boldsymbol\Lambda_{z,\boldsymbol{ij}}(\boldsymbol t)\Big|\stackrel{a.s}{=}0,\\
		&\lim\limits_{\epsilon\rightarrow 0}\lim\limits_{n\rightarrow\infty}\sup_{\boldsymbol t\in(s,\boldsymbol\tau]}\Big| \boldsymbol P_z^{(n,\epsilon)}(\boldsymbol t)- \boldsymbol P_z(\boldsymbol t)\Big| \stackrel{a.s}{=}0.
	\end{align*}
\end{theorem}The proof for this theorem follows at the end of this section.\\
This theorem somewhat describes consistency of the estimators $\boldsymbol\Lambda^{(n,\epsilon)}_{z,\boldsymbol{ij}}$ and $\boldsymbol P_z^{(n,\epsilon)}$. In practice one can fix an $\epsilon$ and let $n$ grow to get a reasonably close estimation of $\boldsymbol\Lambda_{z,\boldsymbol{ij}}$ and $\boldsymbol P_z$.
\textcolor{black}{
In the proof of this theorem we show, among other things, that the solution of the equation \eqref{eq: volterra equation 2} as a functional of $\boldsymbol \Lambda_{z,\boldsymbol i \boldsymbol j}$ and $ \Lambda_{z, i  j}$ is Hadamard continuous. This statement is already known, see Gill, van der Laan \& Wellner \cite{gill1995inefficient}. They did not add a complete proof, so we include this part in our proof. Additionally, they state that this functional is Hadamard differentiable.
}
\begin{remark}
	For the convergence of these landmark estimators we use $\boldsymbol\tau$, meaning $\mathbb{P}(\boldsymbol\tau\leq R)$. In a basic multi-state model with one state process, this can be simplified to $\boldsymbol\tau=(\tau,\tau)$ with $\mathbb{P}(\tau\leq R)$. This is because both jump processes or indicator processes come from the same state-process and therefore generally fall under the same censoring time. We have introduced this freedom of the censoring variable to allow for natural generalization to two jump processes corresponding to two different subjects of examination.
\end{remark}
\begin{remark}
	In the context of actuarial estimation, so called retrospective estimation is important because past developments can be non-measurable regarding the current information $\sigma(\mathcal{G}_s)$, see section \ref{Chap:2}. The definition of the retrospective estimators and the proof of convergence works similar to the prospective estimations, the main difference is the consideration of left truncation. For these reasons, it is out of the scope of this paper.
\end{remark}
\begin{remark}
	\textcolor{black}{
		We introduce the idea that one would follow to prove asymptotic normality, i.e.\ weak convergence to a tight Gaussian process of the introduced estimators. To prove the asymptotic normality of $\boldsymbol\Lambda_{z,\boldsymbol{ij}}^{(n,\epsilon)}(\boldsymbol t)$ and $\boldsymbol{P}_{z}^{(n,\epsilon)}(\boldsymbol t)$  one uses the Hadamard differentiability of the equation \eqref{eq: zwei-dimensionale Kolmogorov Gleichung} interpreted as a functional and ideas introduced in the proof of Theorem 4.5 in Bladt and Furrer \cite{bladt2023conditional}, which are easily generalised to two dimensions, for the demonstration of the asymptotic normality of $\boldsymbol\Lambda_{z,\boldsymbol{ij}}^{(n,\epsilon)}(\boldsymbol t)$.
		Both of these ideas are based on the asymptotic normality of  $\boldsymbol N^{(n)}_{z,\boldsymbol{ij}}(\boldsymbol t)$ and $\boldsymbol{I}_{z,\boldsymbol i}^{(n)}(\boldsymbol t)$, where the second also bases on the first with the help of equation \eqref{I_represent_by_N}. The asymptotic normality of $\boldsymbol N^{(n)}_{z,\boldsymbol{ij}}(\boldsymbol t)$ follows with the bracketing central limit theorem, Theorem 2.11.9 in \cite{VanderLaan1996}. The convergence via the bracketing central limit theorem relies on the bracketing number, which is the minimum number of sets $P_i$ in the partition of $[s,\tau]^2$ such that \begin{align*}
			\mathbb{E}[\sup_{\boldsymbol{t},\boldsymbol{s}\in P_i}|N_{j_1k_1}(t_1\wedge R)N_{j_2k_2}(t_2\wedge R)-N_{j_1k_1}(s_1\wedge R)N_{j_2k_2}(s_1\wedge R)|^2]<\mu^2
		\end{align*}for $\mu$ close to zero. This can be reduced to a problem of one-dimensional partitions, by applying the equality $ab-dc=a(b-c)+c(a-d)$, which allows us to work with differences in one dimension rather than two. This reduces the problem of the bracketing number significantly, since the jump processes are all monotonically increasing and bounded in expectation. From here, one uses the classical arguments for one-dimensional jump processes to follow the asymptotic normality. This conveys the main idea needed for the asymptotic normality of conditional bivariate transition rates and transition probabilities. Since this paper is more concerned with the statistical application of these estimators, a detailed proof is beyond the scope of this paper. }
\end{remark}

\begin{proof}[Proof of Lemma \ref{theo: Konvegenz der erwarteten Sprünge und der Übergangs wahrscheinlichkeiten}]
	From the strong law of large numbers, we immediately get 
	\begin{align}\begin{split}
			&\Big| \boldsymbol N_{z,\boldsymbol{ij}}^{(n)}(\boldsymbol t)-\boldsymbol Q_{z,\boldsymbol{ij}}^c(\boldsymbol t)\Big|\xrightarrow[n\rightarrow\infty]{a.s.} 0,\\
			&\Big| \boldsymbol I_{z,\boldsymbol i}^{(n)}(\boldsymbol t)-\boldsymbol P_{z,\boldsymbol i}^{c}(\boldsymbol t)\Big|\xrightarrow[n\rightarrow\infty]{a.s.} 0\label{eq:application of the law of large numbers},
		\end{split}
	\end{align} for all $\boldsymbol t\in(s,\boldsymbol \tau]$.
	The first step is to generalize this to uniform convergence for $\boldsymbol N_{z,\boldsymbol{ij}}^{(n)}(\boldsymbol t)$. We use an idea that is similar to the idea of the univariate Glivenko-Cantelli theorem. We define the sets$$C_{m,n}:=\left\{\sum_{\boldsymbol i,\boldsymbol j\in\mathcal{Z}^2}|\boldsymbol N_{z,\boldsymbol{ij}}^{(n)}(\boldsymbol \tau)|\leq m\right\}$$ and $C_m:=(C_{m,n} \text{ eventually } )$. This means that $\mathds{1}_{C_m}\boldsymbol N^{(n)}_{z,\boldsymbol{ij}}(\boldsymbol t)$ is eventually bounded on $(s,\boldsymbol \tau]$. Now because of equation \eqref{eq:application of the law of large numbers} we have $\mathbb{P}(C_m)=1$. Thus, we can assume $\boldsymbol N^{(n)}_{z,\boldsymbol{ij}}(t)$ to be bounded eventually for all $\boldsymbol{i}\in\mathcal{Z}^2$.\\
	We now take a partition of $(s,\boldsymbol\tau]$ into finitely many pairwise disjoint rectangles $R_i=(\boldsymbol u_i,\boldsymbol t_i)\subset(s,\boldsymbol \tau]$ and $[s,\tau]=\cup_{i\in I}\overline{R}_i$. With $\d \boldsymbol Q_{z,\boldsymbol{ij}}^c(R_i)\leq \epsilon$ and $\d \boldsymbol Q_{z,\boldsymbol{ij}}^c(\Delta R_i)\leq \epsilon$, where $\d Q_{z,ijkl}^c$ is the one or two-dimensional Lebesgue-Stieltjes measure generated by $Q_{z,ijkl}^c$. This works because for a monotonous function in two-variables, discontinuities can only appear on a countable set of parallels to the $x$- and $y$-axis, see for example the work of Young \& Young \cite{young1924discontinuties}. Thus, for every $\boldsymbol x\in(s,\boldsymbol \tau]$, we get an $a\in I$ with $\boldsymbol x\in R_a=(u_a,t_a)$ or $\boldsymbol x$ is on the edge of one of these rectangles. In the first case, we get \begin{align*}
		\boldsymbol N^{(n)}_{z,\boldsymbol{ij}}(\boldsymbol u_a)\leq \boldsymbol N^{(n)}_{z,\boldsymbol{ij}}(\boldsymbol x)\leq \boldsymbol N^{(n)}_{z,\boldsymbol{ij}}(\boldsymbol t_a^-)\\
		\boldsymbol Q_{z,\boldsymbol{ij}}^{c}(\boldsymbol u_a)\leq \boldsymbol Q_{z,\boldsymbol{ij}}^{c}(\boldsymbol x) \leq \boldsymbol Q_{z,\boldsymbol{ij}}^{c}(\boldsymbol t_a^-)
	\end{align*} 
	and \begin{align*}
		\boldsymbol N^{(n)}_{z,\boldsymbol{ij}}(\boldsymbol u_a)-\boldsymbol Q_{z,\boldsymbol{ij}}^{c}(\boldsymbol t_a^-)\leq \boldsymbol N^{(n)}_{z,\boldsymbol{ij}}(\boldsymbol x)-\boldsymbol Q_{z,\boldsymbol{ij}}^{c}(\boldsymbol x)\leq \boldsymbol N^{(n)}_{z,\boldsymbol{ij}}(\boldsymbol t_a^-)-\boldsymbol Q_{z,\boldsymbol{ij}}^{c}(\boldsymbol u_a).
	\end{align*}Adding $0$ we get \begin{align*}
		\boldsymbol N^{(n)}_{z,\boldsymbol{ij}}(\boldsymbol u_a)-\boldsymbol Q_{z,\boldsymbol{ij}}^{c}(\boldsymbol u_a)+\boldsymbol Q_{z,\boldsymbol{ij}}^{c}(\boldsymbol u_a)-\boldsymbol Q_{z,\boldsymbol{ij}}^{c}(\boldsymbol t_a^-)\leq \boldsymbol N^{(n)}_{z,\boldsymbol{ij}}(\boldsymbol x)-\boldsymbol Q_{z,\boldsymbol{ij}}^{c}(\boldsymbol x)\\
		\boldsymbol N^{(n)}_{z,\boldsymbol{ij}}(\boldsymbol t_a^-)-\boldsymbol Q_{z,\boldsymbol{ij}}^{c}(\boldsymbol t_a^-)+\boldsymbol Q_{z,\boldsymbol{ij}}^{c}(\boldsymbol t_a^-)-\boldsymbol Q_{z,\boldsymbol{ij}}^{c}(\boldsymbol u_a)\geq \boldsymbol N^{(n)}_{z,\boldsymbol{ij}}(\boldsymbol x)-\boldsymbol Q_{z,\boldsymbol{ij}}^{c}(\boldsymbol x).
	\end{align*}The last step is, to evaluate $\boldsymbol Q_{z,\boldsymbol{ij}}^{c}(\boldsymbol t_a^-)-\boldsymbol Q_{z,\boldsymbol{ij}}^{c}(s)$ on every rectangle $R_i$. We get\begin{align*}
		\Big|\boldsymbol Q_{z,\boldsymbol{ij}}^{c}(t_{a,1}^-,t_{a,2}^-)&-\boldsymbol Q_{z,\boldsymbol{ij}}^{c}(u_{a,1},u_{a,2})\Big|\\
		&\leq \Big|\boldsymbol Q_{z,\boldsymbol{ij}}^{c}(u_{a,1},t_{a,2}^-)-\boldsymbol Q_{z,\boldsymbol{ij}}^{c}(u_{a,1},u_{a,2})\Big|\\
		&\quad+ \Big|\boldsymbol Q_{z,\boldsymbol{ij}}^{c}(t_{a,1}^-,u_{a,2})-\boldsymbol Q_{z,\boldsymbol{ij}}^{c}(u_{a,1},u_{a,2})\Big|\\
		&\quad+\Big|\boldsymbol Q_{z,\boldsymbol{ij}}^{c}(\boldsymbol t_a^-)-\boldsymbol Q_{z,\boldsymbol{ij}}^{c}(t_{a,1}^-,u_{a,2})-\boldsymbol Q_{z,\boldsymbol{ij}}^{c}(u_{a,1},t_{a,2}^-)+\boldsymbol Q_{z,\boldsymbol{ij}}^{c}(\boldsymbol u_a)\Big|\\
		&\leq 3\epsilon.
	\end{align*}All in all we get the following assertion \begin{align*}
		\boldsymbol N^{(n)}_{z,\boldsymbol{ij}}(\boldsymbol u_a)-\boldsymbol Q_{z,\boldsymbol{ij}}^{c}(\boldsymbol u_a)-3\epsilon\leq \boldsymbol N^{(n)}_{z,\boldsymbol{ij}}(\boldsymbol x)-\boldsymbol Q_{z,\boldsymbol{ij}}^{c}(\boldsymbol x)\\
		\boldsymbol N^{(n)}_{z,\boldsymbol{ij}}(\boldsymbol t_a^-)-\boldsymbol Q_{z,\boldsymbol{ij}}^{c}(\boldsymbol t_a^-)+3\epsilon\geq \boldsymbol N^{(n)}_{z,\boldsymbol{ij}}(\boldsymbol x)-\boldsymbol Q_{z,\boldsymbol{ij}}^{c}(\boldsymbol x).
	\end{align*}Now $\forall \epsilon>0$ take the previously defined partition $(R_i)_{i\in\mathcal{Z}}$ and define $N(\epsilon)$ such that $\boldsymbol N^{(n)}_{z,\boldsymbol{ij}}(\boldsymbol u_i)-\boldsymbol Q_{z,\boldsymbol{ij}}^{c}(\boldsymbol u_i)>-\epsilon$ and $\boldsymbol N^{(n)}_{z,\boldsymbol{ij}}(\boldsymbol t_i-)-\boldsymbol Q_{z,\boldsymbol{ij}}^{c}(\boldsymbol t_i-)<\epsilon$ for all $R_i$. This is possible due to the finite nature of the partition and \eqref{eq:application of the law of large numbers}. Then we get $\boldsymbol N^{(n)}_{z,\boldsymbol{ij}}(\boldsymbol x)-\boldsymbol Q_{z,\boldsymbol{ij}}^{c}(\boldsymbol x)<4\epsilon$.
	For the second case, where $\boldsymbol x$ is on the edge of one of the rectangles, we can apply the one-dimensional Glivenko-Cantelli idea on a parallel to the $x$- or $y$-axis and get to the same outcome. 
	Thus, $$\sup_{\boldsymbol x\in[s,\boldsymbol \tau]}\Big|\boldsymbol N^{(n)}_{z,\boldsymbol{ij}}(\boldsymbol x)-\boldsymbol Q_{z,\boldsymbol{ij}}^{c}(\boldsymbol x)\Big|\xrightarrow[n\rightarrow\infty]{}0$$
	on $C_{m}$. Referring to the fact that $\boldsymbol Q_{z,\boldsymbol{ij}}^c(\tau)$ is bounded and that $\boldsymbol N_{z,\boldsymbol{ij}}^{(n)}(\boldsymbol\tau)\xrightarrow[n\rightarrow\infty]{a.s.}\boldsymbol Q_{z,\boldsymbol{ij}}^c(\boldsymbol \tau)$, $\mathbb{P}(C_m)=1$ for $m$ large enough. All in all, we get the almost sure uniform convergence of $\boldsymbol N_{z,\boldsymbol{ij}}^{(n)}(\boldsymbol t)$.\\
	For the uniform convergence of $\boldsymbol I_{z,\boldsymbol i}^{(n)}(\boldsymbol t)$ we use the fact that we can calculate the state-indicator process from the full information over all the jump processes and the information at the current time $s$ and then transfer the uniform convergence of the conditional transition rates to the transition probabilities. This is a common idea for the convergence of transition probabilities, see for example the work of Bladt \& Furrer \cite{bladt2023conditional}. In our case we have the formula \begin{align*}
		I_{i_1}(t_1)I_{i_2}(t_2)&=I_{i_1}(s)I_{i_2}(s)+I_{i_2}(s)\sum_{\substack{k\in\mathcal{Z}}}N_{ki_1}(t_1)+I_{i_1}(s)\sum_{\substack{k\in\mathcal{Z}}}N_{ki_2}(t_2)\\
		&\quad+\sum_{\substack{\boldsymbol k\in\mathcal{Z}^2}}N_{k_1i_1}(t_1)N_{k_2i_2}(t_2).
	\end{align*}Because of the censoring we need to adjust this formula to \begin{align*}
		I_{i_1}(t_1)I_{i_2}(t_2)\mathds{1}_{\{\boldsymbol t<R\}}&=I_{i_1}(s)I_{i_2}(s)+I_{i_2}(s)\sum_{\substack{k\in\mathcal{Z}}}N_{ki_1}(t_1\wedge R)\\
		&\quad +I_{i_1}(s)\sum_{\substack{k\in\mathcal{Z}}}N_{ki_2}(t_2\wedge R)
		+\sum_{\substack{\boldsymbol k\in\mathcal{Z}^2}}N_{k_1i_1}(t_1\wedge R)N_{k_2i_2}(t_2\wedge R)\\
		&\quad-\mathds{1}_{\{t_1\geq R\vee t_2\geq R\}}I_{i_1}(t_1\wedge R)I_{i_2}(t_2\wedge R\texttt{}).
	\end{align*}This means that the estimator $I_{z,ij}^{(n)}(\boldsymbol t)$ satisfies the following equation: \begin{align}
		\begin{split}
			\boldsymbol I_{z,\boldsymbol i}^{(n)}(\boldsymbol t)&=I_{z,\boldsymbol i}^{(n)}(s)+\boldsymbol I_{z,i_1}^{(n)}(s)\sum_{\substack{k\in\mathcal{Z}}}{N}^{(n)}_{z,ki_2}(t_2)+I_{z,i_2}^{(n)}(s)\sum_{\substack{k\in\mathcal{Z}}}{N}^{(n)}_{z,ki_1}(t_1)\\
			&\quad+ \sum_{\substack{\boldsymbol k\in\mathcal{Z}^2}}\boldsymbol N^{(n)}_{z,\boldsymbol{ki}}(\boldsymbol t)- C_{z,\boldsymbol i}^{(n)}(\boldsymbol t),
		\end{split}\label{eq:Zerlegung des Schätzers Für die Übergangswahrscheinlichkeiten}
	\end{align} where \begin{align*} C_{z,\boldsymbol i}^{(n)}(\boldsymbol t):=\frac{1}{n}\sum_{m=1}^{n}\mathds{1}_{\{R^m\leq t_1\vee R^m\leq t_2\}}\mathds{1}_{\{Z^m_{t_1\wedge R^m}=i_1,Z^m_{t_2\wedge R^m}=i_2\}}\mathds{1}_{\{\xi^m=z\}},\end{align*} for $\boldsymbol i\in\mathcal{Z}^2$.
	Analogously, we get an equality for the occupation probabilities: \begin{align}
		\begin{split}
			\boldsymbol P_{z,\boldsymbol i}^{c}(\boldsymbol t)&=\boldsymbol P_{z,\boldsymbol i}^{c}(s)+P_{z,i_1}^{c}(s)\sum_{\substack{k\in\mathcal{Z}}}Q^{c}_{z,ki_2}(t_2)+P_{z,i_2}^{c}(s)\sum_{\substack{k\in\mathcal{Z}}}Q^c_{z,ki_1}(t_1)\\
			&\quad+ \sum_{\substack{\boldsymbol k\in\mathcal{Z}^2}}\boldsymbol Q^c_{z,\boldsymbol{ki}}(\boldsymbol t)- C_{z,\boldsymbol i}(\boldsymbol t),\end{split}\label{eq:Zerlegung der Übergangswahrscheinlichkeiten}
	\end{align}where $ C_{z,\boldsymbol i}(\boldsymbol t):=\mathbb{P}(t_1\geq R\vee t_2\geq R, Z_{t_1\wedge R}=i_1,Z_{t_2\wedge R}=i_2,\xi=z)$ for $\boldsymbol i\in\mathcal{Z}^2$. \\
	The next step is to show $\sup_{\boldsymbol t\in(s,\boldsymbol \tau]}\Big| C_{z,\boldsymbol i}(\boldsymbol t)- C_{z,\boldsymbol i}^{(n)}(\boldsymbol t)\Big|\xrightarrow[n\rightarrow\infty]{a.s.}0$. We have\begin{align*}
		&I_{i_1}(t_1\wedge R)I_{i_2}(t_2\wedge R)\mathds{1}_{\{R\leq t_1\vee R\leq t_2\}}\\
		&=\bigg[I_{i_1}(s)I_{i_2}(s)+I_{i_2}(s)\sum_{k\in\mathcal{Z}}N_{ki_1}(t_1\wedge R)+I_{i_1}(s)\sum_{k\in\mathcal{Z}}N_{ki_2}(t_2\wedge R)
		\\&\quad+\sum_{\boldsymbol k\in\mathcal{Z}^2}N_{k_1i_1}(t_1\wedge R)N_{k_2i_2}(t_2\wedge R)\bigg]	\mathds{1}_{\{R\leq t_1\vee R\leq t_2\}}.
	\end{align*}Similar to the approach with $\boldsymbol P_{z,\boldsymbol i}(\boldsymbol t)$ and $\boldsymbol I_{z,\boldsymbol i}^{(n)}(\boldsymbol t)$, we can now separate $C_{z,\boldsymbol i}(\boldsymbol t)$ and $\boldsymbol C_{z,\boldsymbol i}^{(n)}(\boldsymbol t)$ into four summands that are monotonically increasing and bounded on a set with probability $1$. Thus, with the same two-dimensional Glivenko-Cantelli arguments we used before and the strong law of large numbers, we get the almost sure uniform convergence of $\boldsymbol C_{z,\boldsymbol i}^{(n)}(\boldsymbol t)$ to $\boldsymbol C_{z,\boldsymbol i}(\boldsymbol t)$. With equations \eqref{eq:Zerlegung der Übergangswahrscheinlichkeiten}, \eqref{eq:Zerlegung des Schätzers Für die Übergangswahrscheinlichkeiten}, the uniform convergence of the univariate estimators 
	and the law of large numbers, we get the almost sure uniform convergence of $\boldsymbol I_{z,\boldsymbol i}^{(n)}$ to $\boldsymbol P_{z,\boldsymbol i}^{c}$ on $(s,\boldsymbol \tau]$.
\end{proof}
\begin{proof}[Proof of Theorem \ref{theo: convergence of Lambda and P}]
	For the proof of this theorem, we want to apply the continuity of the functional \begin{align*}
		L:(F,G)\mapsto\int\displaylimits_{(s,\boldsymbol t]}\frac{1}{F(u)}\d G(u),
	\end{align*}see Lemma \ref{lem:hadamard diff} for the precise definition of continuity and the contributing function spaces. For this lemma, we need that $\frac{1}{\boldsymbol I_{\boldsymbol i}^{(n)}\vee \epsilon}$, $\boldsymbol N_{z,\boldsymbol{ij}}^{(n)}$, $\frac{1}{\boldsymbol P_{z,\boldsymbol i}^c\vee \epsilon}$, and $\boldsymbol Q_{z,\boldsymbol{ij}}^c$ are all of almost surely uniform bounded variation in $n$, see the Appendix for the definition of the variation \ref{def:variation}. \\
For $\boldsymbol N_{z,\boldsymbol{ij}}^{(n)}$ and $\boldsymbol Q_{z,\boldsymbol{ij}}^c$, the almost sure uniform bounded variation follows directly from the monotonicity and the fact that $\boldsymbol N_{z,\boldsymbol{ij}}^{(n)}(\tau)$ is almost surely bounded, see the proof of Lemma \ref{theo: Konvegenz der erwarteten Sprünge und der Übergangs wahrscheinlichkeiten}.
	For the uniform bounded variation of $\frac{1}{\boldsymbol I_{\boldsymbol i}^{(n)}\vee \epsilon}$ and $\frac{1}{\boldsymbol P_{z,\boldsymbol i}^c\vee \epsilon}$, we use the fact that $$\Big\Vert\frac{1}{f(\boldsymbol x)\vee \epsilon}\Big\Vert_{v_1}^{(2)}\leq \frac{1}{\epsilon}+\frac{1}{\epsilon^4}\Vert f(\boldsymbol x)\Vert_{v_1}^{(2)}.$$
	Now $\Vert{\boldsymbol I_{\boldsymbol i}^{(n)}}\Vert_{v_1}^{(2)}$ and $\Vert{\boldsymbol P_{z,\boldsymbol i}^c}\Vert_{v_1}^{(2)}$ are both of almost surely bounded variation by equation \eqref{I_represent_by_N} and the same techniques as before. With Lemma \ref{lem:hadamard diff} the statement follows immediately.\\
 The next step is the convergence of $\boldsymbol P_{\boldsymbol i}^{(n,\epsilon)}(\boldsymbol t)$. We first need to take care of the problem that $\boldsymbol P_{\boldsymbol i}^{(n,\epsilon)}(\boldsymbol t)$ is defined as a solution to an integral equation. For that we rewrite the integral equation.
	We start by defining $\tilde  P^{(\epsilon)}(\boldsymbol t)$:  \begin{align*}
		 \tilde P^{(\epsilon)}_{l(i-1)+k}(\boldsymbol{t}):=P^{(\epsilon)}_{ki}(\boldsymbol t).
		\end{align*}This $\tilde  P^{(\epsilon)}(\boldsymbol t)$ corresponds to a column by column translation of the matrix $\boldsymbol P^{(\epsilon)}(\boldsymbol t)$. Now rearrange $\boldsymbol \Lambda_{z,\boldsymbol{ij}}(\boldsymbol t)$ into the matrix $\Lambda^{(\epsilon)}$, such that $\tilde  P^{(\epsilon)}(\boldsymbol t)$ is a solution to the inhomogeneous Volterra equation \begin{align}
		Y(\boldsymbol t)=\phi^{(\epsilon)}(\boldsymbol t)+\int\displaylimits_{( s,\boldsymbol t]}Y(\boldsymbol u^-)\Lambda^{(\epsilon)}(\d \boldsymbol u),\label{eq:two-dimensional vector equation}
	\end{align}
		where $\phi^{(\epsilon)}$ corresponds to the one-dimensional part of the two-dimensional integral equation  \begin{align*}
						\phi^{(\epsilon)}(\boldsymbol t)_{l(i_2-1)+i_1}&:=-P^{(\epsilon)}_{i_2}(s)P^{(\epsilon)}_{i_1}(s)+P^{(\epsilon)}_{k}(t_1)P^{(\epsilon)}_{i_1}(s)+P^{(\epsilon)}_{i_2}(s)P^{(\epsilon)}_{i_1}(t_2),
			\end{align*} with $l=\#\mathcal{Z}$. This works because for $\boldsymbol t> s$, the calculation of $P_{\boldsymbol i}(\boldsymbol t)$ only relies on $\boldsymbol\Lambda_{\boldsymbol{ji}}(\boldsymbol t)_{\boldsymbol j\in\mathcal{Z}^2}$, thus, we can rearrange $\boldsymbol P_{\boldsymbol i}$ and $\boldsymbol\Lambda_{\boldsymbol{ji}}(\boldsymbol t)_{\boldsymbol j\in\mathcal{Z}^2}$ into a vector and a matrix. This transformed the rather complicated integral equation \eqref{2dimKolmForwEq} into the well known and understood inhomogeneous Volterra equation.\\		
From Lemma \ref{lem Peano as solution} we know that the solution to this inhomogeneous Volterra equation is \begin{align*}
		\tilde{\boldsymbol P}^{(\epsilon)}(\boldsymbol t)&=\phi^{(\epsilon)}(\boldsymbol t)+\int\displaylimits_{(s,\boldsymbol t]}\phi^{(\epsilon)}(\boldsymbol u^-)\boldsymbol\Lambda^{(\epsilon)}(\d \boldsymbol u)\mathcal{P}((\boldsymbol u,\boldsymbol t],\boldsymbol\Lambda^{(\epsilon)}),
\end{align*} with \begin{align*}
		\mathcal{P}((s,\boldsymbol t],\Lambda)&:=Id+\sum\displaylimits_{n=1}^\infty\;\; \idotsint\displaylimits_{s<\boldsymbol u^{(1)}<...<\boldsymbol u^{(n)}\leq t}\;\;\boldsymbol\Lambda(\d \boldsymbol u^{(1)})\cdots\boldsymbol\Lambda(\d \boldsymbol u^{(n)}).
\end{align*}
Using this solution, we can calculate the difference between the solution of the estimated integral equation and the solution to the theoretical integral equation, where we use $\boldsymbol{\tilde P} ^{(n,\epsilon)}(\boldsymbol t)$ as the solution to the inhomogeneous Volterra equation with estimated $\boldsymbol \Lambda^{(n,\epsilon)}$ and $\phi^{(n,\epsilon)}$. We divide the estimation error into the estimation error of the integrand and the integrator.
	\begin{align}
		\sup_{\boldsymbol t\in(s,\boldsymbol\tau]}\Big\Vert &\tilde {\boldsymbol{P}}^{(\epsilon)}( \boldsymbol t)-\tilde{\boldsymbol{ P}}^{(n,\epsilon)}(\boldsymbol t)\Big\Vert_\infty\nonumber\\
		&\leq\sup_{\boldsymbol t\in(s,\boldsymbol\tau]} \Big\Vert\phi^{(\epsilon)}(\boldsymbol t)-\phi^{(n,\epsilon)}(\boldsymbol t)\Big\Vert\nonumber\nonumber\\
		&\quad+\sup_{\boldsymbol t\in(s,\boldsymbol\tau]}\Big\Vert\int\displaylimits_{(s,\boldsymbol t]}\left(\phi^{(\epsilon)}(\boldsymbol u^-)-\phi^{(n,\epsilon)}(\boldsymbol u^-)\right)\boldsymbol\Lambda^{(n,\epsilon)}(\d \boldsymbol u)\mathcal{P}((\boldsymbol u,\boldsymbol t],\boldsymbol\Lambda^{(n,\epsilon)})\Big\Vert_\infty\label{convergence of P first formula}\\
		 &\quad+\sup_{\boldsymbol t\in(s,\boldsymbol\tau]}\Big\Vert \int\displaylimits_{(s,\boldsymbol t]} \phi^{(\epsilon)}(\boldsymbol u^-)\Big[\Lambda^{(\epsilon)}( \d \boldsymbol u)\mathcal{P}((\boldsymbol u,\boldsymbol t],\boldsymbol\Lambda^{(\epsilon)})-\nonumber
		 \\
		 &\;\;\;\;\;\;\;\;\;\;\;\;\;\;\;\;\;\,\;\;\;\;\;\;\;\;\;\;\;\;\;\;\;\;\;\;\;\,\;\;\;\;\;\;\;\;\;\;\;\;\;\;\;\;\;\;\;\,\;\;\;\;\;\;\boldsymbol\Lambda^{(\epsilon,n)}( \d \boldsymbol u)\mathcal{P}((\boldsymbol u,\boldsymbol t],\boldsymbol\Lambda^{(\epsilon,n)})\Big]\Big\Vert_\infty\label{convergence of P second formula}
	\end{align}
The first two summands \eqref{convergence of P first formula} can be taken care of by simple monotonicity arguments: \begin{align*}
	&\sup_{\boldsymbol t\in(s,\boldsymbol\tau]}\Big\Vert\phi^{(\epsilon)}(\boldsymbol t)-\phi^{(n,\epsilon)}(\boldsymbol t)\Big\Vert_\infty\;\\&\;\;\;\;\;\;\;\;\;\;\;\;\;\;\;\;\;\,\;\;+\sup_{\boldsymbol t\in(s,\boldsymbol\tau]}\Big\Vert\int\displaylimits_{(s,\boldsymbol t]}\phi^{(\epsilon)}(\boldsymbol u^-)-\phi^{(n,\epsilon)}(\boldsymbol u^-)\boldsymbol\Lambda^{(n,\epsilon)}(\d \boldsymbol u)\mathcal{P}((\boldsymbol u,\boldsymbol t],\boldsymbol\Lambda^{(n,\epsilon)})\Big\Vert_\infty\\
	&\;\;\;\;\;\;\;\;\;\;\;\;\;\;\;\;\;\,\;\;\leq \sup_{\boldsymbol t\in(s,\boldsymbol\tau]}\Big\Vert\phi^{(\epsilon)}(\boldsymbol t)-\phi^{(n,\epsilon)}(\boldsymbol t)\Big\Vert_\infty\left(2+\Big\Vert\mathcal{P}((s,\boldsymbol t],\boldsymbol\Lambda^{(n,\epsilon)})\Big\Vert_\infty\right).
\end{align*}
For the second summand \eqref{convergence of P second formula} we again segment the statistical error for the different parts of the integral. We get \begin{align}
	\sup_{\boldsymbol t\in(s,\boldsymbol\tau]}\Big\Vert \int\displaylimits_{(s,\boldsymbol t]} \phi^{(\epsilon)}(\boldsymbol u^-)\left[\boldsymbol\Lambda^{(\epsilon)}( \d \boldsymbol u)\mathcal{P}((\boldsymbol u,\boldsymbol t],\boldsymbol\Lambda^{(\epsilon)})-\boldsymbol\Lambda^{(\epsilon,n)}( \d \boldsymbol u)\mathcal{P}((\boldsymbol u,\boldsymbol t],\boldsymbol\Lambda^{(\epsilon,n)})\right]\Big\Vert_\infty\nonumber\\
	\leq \sup_{\boldsymbol t\in(s,\boldsymbol\tau]}\Big\Vert \int\displaylimits_{(s,\boldsymbol t]} \phi^{(\epsilon)}(\boldsymbol u^-)\boldsymbol\Lambda^{(\epsilon)}( \d \boldsymbol u)\left[\mathcal{P}((\boldsymbol u,\boldsymbol t],\boldsymbol\Lambda^{(\epsilon)})-\mathcal{P}((\boldsymbol u ,\boldsymbol t],\boldsymbol\Lambda^{(\epsilon,n)})\right]\Big\Vert_\infty\label{convergence of P third summand}\\
	+\sup_{\boldsymbol t\in(s,\boldsymbol\tau]}\Big\Vert \int\displaylimits_{(s,\boldsymbol t]} \phi^{(\epsilon)}(\boldsymbol u^-)\left[\boldsymbol\Lambda^{(\epsilon)}( \d \boldsymbol u)-\boldsymbol\Lambda^{(\epsilon,n)}( \d \boldsymbol u)\right]\mathcal{P}((\boldsymbol u,\boldsymbol t],\boldsymbol\Lambda^{(\epsilon,n)})\Big\Vert_\infty\label{convergence of P fourth summand}
\end{align}  For the second summand \eqref{convergence of P fourth summand}, we use two-dimensional integration by parts and the fact, that $\mathcal{P}((\boldsymbol u,\boldsymbol t],\boldsymbol\Lambda^{(\epsilon,n)})$ is monotone in $\boldsymbol u$ and $\boldsymbol t$ and thus, has a finite two-dimensional variation,  and that with the definition of $P_i^{(\epsilon)}$, the definition of $\Vert \cdot\Vert_\infty$, and (37) from Christiansen \& Furrer \cite{christiansen2022extension}, we have \begin{align*}\Big\Vert\big\Vert \phi^{(\epsilon)}\big\Vert_{v_1}^{(2)}\Big\Vert_\infty\leq\left[\Big\Vert\sum_{i\in\mathcal{Z}}\big\Vert P_{i}^{(\epsilon)}\big\Vert_{v_1}^{(1)}\Big\Vert_\infty\right]^2\leq\left[\Big\Vert\big\Vert P(s)\Prodi_{(s,t]}(Id+\Lambda^{(\epsilon)}(\d u))\big\Vert_{v_1}^{(1)}\Big\Vert_\infty\right]^2<\infty.\end{align*} All in all, we find an upper bound of the form $$ C(n) \sup_{\boldsymbol t\in(s,\boldsymbol\tau]}\Big\Vert\boldsymbol\Lambda^{(\epsilon)}(\boldsymbol t)-\boldsymbol\Lambda^{(n,\epsilon)}(\boldsymbol t)\Big\Vert_\infty, $$ where $C(n)$ contains, the variation of $\phi^{(\epsilon)}$, and $\mathcal{P}(( s,\boldsymbol \tau],\Lambda^{(\epsilon,n)})$, which is all almost surely eventually bounded in $n$. \\
For the first summand \eqref{convergence of P third summand} we get with the monotonicity of $\Lambda^{(\epsilon)}$: \begin{align*}
	&\sup_{\boldsymbol t\in(s,\boldsymbol\tau]}\Big\Vert \int\displaylimits_{(s,\boldsymbol t]} \phi^{(n,\epsilon)}(\boldsymbol u^-)\Lambda^{(\epsilon)}( \d\boldsymbol u)\left[\mathcal{P}((\boldsymbol u,\boldsymbol t],\Lambda^{(\epsilon)})-\mathcal{P}((\boldsymbol u,\boldsymbol t],\Lambda^{(\epsilon,n)})\right]\Big\Vert_\infty\\
	&\leq 	\sup_{\boldsymbol t\in(s,\boldsymbol\tau]}\sup_{\boldsymbol u\in(s,\boldsymbol \tau]}\Big\Vert\mathcal{P}((\boldsymbol u,\boldsymbol t],\Lambda^{(\epsilon)})-\mathcal{P}((\boldsymbol u,\boldsymbol t],\Lambda^{(\epsilon,n)})\Big\Vert_\infty\sup_{\boldsymbol u\in(s,\boldsymbol \tau]}\Big\Vert\phi^{(n,\epsilon)}(\boldsymbol u)\Big\Vert_\infty
	\sup_{\boldsymbol u\in(s,\boldsymbol \tau]}\Big\Vert\Lambda^{(\epsilon)}(\boldsymbol u)\Big\Vert_\infty.
\end{align*}Thus, we need to evaluate $\Big\Vert\mathcal{P}((\boldsymbol u,\boldsymbol t],\Lambda^{(\epsilon)})-\mathcal{P}((\boldsymbol u,\boldsymbol t],\Lambda^{(\epsilon,n)})\Big\Vert_\infty$. By applying the Duhamel equality, see Lemma \ref{lem: Duhamel}, we get \begin{align*}
\Big\Vert&\mathcal{P}((\boldsymbol u,\boldsymbol t],\Lambda^{(\epsilon)})-\mathcal{P}((\boldsymbol u,\boldsymbol t],\Lambda^{(\epsilon,n)})\Big\Vert_\infty\\
&=\Big\Vert\int\displaylimits_{(\boldsymbol u,\boldsymbol t]}\mathcal{P}((\boldsymbol u,\boldsymbol t]\cap ( 0,s),\Lambda^{(\epsilon)})(\Lambda^{(\epsilon)}(\d s)-\Lambda^{(\epsilon,n)}(\d s))\mathcal{P}((\boldsymbol u,\boldsymbol t]\cap (s,\infty),\Lambda^{(\epsilon,n)})\Big\Vert_\infty.
\end{align*}Again, we can apply two-dimensional integration by parts and find that there is a constant $C(n)$ which is almost surely bounded eventually and satisfies \begin{align*}
\Big\Vert\int\displaylimits_{(\boldsymbol u,\boldsymbol t]}\mathcal{P}((\boldsymbol u,\boldsymbol t]\cap (0,s),\Lambda^{(\epsilon)})(\Lambda^{(\epsilon)}(\d s)-\Lambda^{(\epsilon,n)}(\d s))\mathcal{P}((\boldsymbol u,\boldsymbol t]\cap (s,\infty),\Lambda^{(\epsilon,n)})\Big\Vert_\infty\\
\leq C(n) \sup_{\boldsymbol t\in(s,\boldsymbol\tau]}\left\Vert\Lambda^{(\epsilon)}(\boldsymbol t)-\Lambda^{(n,\epsilon)}(\boldsymbol t)\right\Vert_\infty.
\end{align*}This $C(n)$ contains the variation of $\mathcal{P}((\boldsymbol u,\boldsymbol t]\cap (0,s),\Lambda^{(\epsilon)})$ and of $\mathcal{P}((\boldsymbol u,\boldsymbol t]\cap (s,\infty),\Lambda^{(\epsilon,n)})$ which is again almost surely bounded eventually, because they are both monotone in $\boldsymbol u$ and $\boldsymbol t$. Combining all of this, we get the desired assertion: $$	\sup_{\boldsymbol t\in(s,\boldsymbol\tau]}\left\Vert \tilde P^{(\epsilon)}(\boldsymbol t)-\tilde P^{(n,\epsilon)}(\boldsymbol t)\right\Vert_\infty \xrightarrow[n\rightarrow\infty]{a.s.}0.$$
Now we have the convergence in $n$ of the estimators to $\boldsymbol \Lambda^{(\epsilon)}_{z,\boldsymbol{ij}}(\boldsymbol t)$ and $ \boldsymbol P_{z,i}^{(\epsilon)}(\boldsymbol t)$. The next step is the convergence in $\epsilon$.\\
	For the convergence of $\boldsymbol\Lambda^{(\epsilon)}_{z,\boldsymbol{ij}}$ we start by showing pointwise convergence. We apply the monotone convergence theorem on $\boldsymbol \Lambda_{z,\boldsymbol{ij}}^{(\epsilon)}$. This works because of \eqref{IntegrabilCond} and equation (35) of Christiansen \cite{christiansen2021calculation}. We get \begin{align*}
	\int\displaylimits_{(s,\boldsymbol t]}\frac{1}{\boldsymbol P_{z,\boldsymbol i}^c(\boldsymbol u^-)\vee \epsilon}\boldsymbol Q_{z,\boldsymbol{ij}}^c(\d  \boldsymbol u)\xrightarrow[\epsilon\rightarrow0]{a.s.}\int\displaylimits_{(s, \boldsymbol t]}\frac{1}{\boldsymbol P_{z,\boldsymbol i}^c( \boldsymbol u^-)}\boldsymbol Q_{z,\boldsymbol{ij}}^c(\d \boldsymbol u)=\int\displaylimits_{(s,\boldsymbol t]}\frac{\mathds{1}_{\{\boldsymbol P_{z,\boldsymbol i}^c(\boldsymbol u^-)>0\}}}{\boldsymbol P_{\boldsymbol z,\boldsymbol i}^c(\boldsymbol u^-)}\boldsymbol Q_{z,\boldsymbol {ij}}^c(\d \boldsymbol u),
\end{align*}for $\boldsymbol t\in(s,\boldsymbol \tau]$. Thus, $\boldsymbol\Lambda_{z,\boldsymbol{ij}}^{(\epsilon)}(\boldsymbol t)\xrightarrow[\epsilon\rightarrow0]{a.s.}\boldsymbol \Lambda_{z,\boldsymbol {ij}}(\boldsymbol t),$ for all $ \boldsymbol t\in(s,\boldsymbol \tau]$. Now $\boldsymbol\Lambda_{z,\boldsymbol{ij}}^{(\epsilon)}$ and $\boldsymbol\Lambda_{z,\boldsymbol{ij}}$ are both monotonically increasing and bounded on $(s,\boldsymbol \tau]$. With arguments similar to the two-dimensional proof of the Gliveno-Cantelli theorem the pointwise convergence can be extended to $$\sup_{ \boldsymbol t\in(s,\boldsymbol\tau]}\Big| \boldsymbol \Lambda_{z,\boldsymbol{ij}}^{(\epsilon)}(\boldsymbol t)- \boldsymbol\Lambda_{z,ij}(\boldsymbol t)\Big| \xrightarrow[\epsilon\rightarrow0]{a.s.}0.$$
The same argumentation works for the univariate conditional transition rates. The application of the monotone convergence theorem proceeds in the same way as the bivariate case.\\
 For the convergence of transition probabilities, we use the same argumentation we used in the convergence in $n$, namely integration by parts and Duhamel equality.\\
 Now we have \begin{align*}
		\sup_{\boldsymbol t\in (s,\tau]}\Big|\boldsymbol\Lambda_{z,\boldsymbol{ij}}^{(n,\epsilon)}(\boldsymbol t)-\boldsymbol\Lambda_{z,\boldsymbol{ij}}(\boldsymbol t)\Big|
		\leq \sup_{\boldsymbol t\in (s,\tau]}\Big|\boldsymbol\Lambda_{z,\boldsymbol{ij}}^{(n,\epsilon)}(\boldsymbol t)-\boldsymbol\Lambda_{z,\boldsymbol{ij}}^{(\epsilon)}(\boldsymbol t)\Big|
		+\sup_{\boldsymbol t\in (s,\tau]}\Big|\boldsymbol\Lambda_{z,\boldsymbol{ij}}^{(\epsilon)}(\boldsymbol t)-\boldsymbol\Lambda_{z,\boldsymbol{ij}}(\boldsymbol t)\Big|.
	\end{align*}This directly concludes the proof of the statement.
\end{proof}
\begin{proof}[Sketch of the proof for Theorem \ref{theo:convergence of Lambda and P 1D}]
	The proof uses convergences in \eqref{eq: Konvergenzen in 1-D} and ideas that are analogous to the ideas that were used in the proof of the $\epsilon$ convergence of Theorem \ref{theo: convergence of Lambda and P}.
\end{proof}
\section{Uniform Convergence in Actuarial Estimation}\label{Chap: Section 5}
\textcolor{black}{Similar to the general multi-state framework, we define $Z=(Z(t))_{t\geq s},$ as the state process of the insured with the state indicator process $(I_{i})_{i\in\mathcal{Z}}$ and the transition counting process $(N_{ij})_{i,j\in\mathcal{Z}}$. The landmark $\xi=Z(s)$, similar to the information model used in Markov modelling in the actuarial context, is realistic in application, but different conditioning is also feasible.}\\ Let $B$ be the insurance cash flow of an individual life insurance contract. We assume $B$ to be an adapted c\`{a}dl\`{a}g process with paths of finite variation. For modelling purposes we use a maximum contract horizon of $T$, which means that
\begin{align*}
	B(\d t) = 0 , \quad t > T.
\end{align*}
\noindent
\textcolor{black}{
  Several models have been proposed to estimate the expected value of these future cash flows.  The Markov model only allows for restrictive payment functions in the future cash flow $B(t)$ and requires the Markov assumption. The more general semi-Markov model allows for the estimation of cash flows with options as policyholder behaviour, see Ahmad et al.\ in \cite{Ahmad_Buchardt_Furrer_2022}. The disadvantage of this model is that it still introduces systematic model risk because we cannot verify the required Markov assumption and that the estimation is computationally expensive. Therefore, modelling without the Markov assumption in a non-Markov model is the more general way to model future cash flows. This has been done for scaled cash flows by Christiansen and Furrer in \cite{christiansen2022extension}. There they use an approach that uses contract-dependent estimates to account for policyholder behaviour. The disadvantage is that, in practice, re-estimating for different insurance contracts is time-consuming and introduces other problems, particularly in terms of data usage and storage. \\
 The conditional bivariate transition and transition probabilities introduced in this paper, allow for an estimation procedure that is independent of the design of the insurance contract and allows for a wide variety of payment functions. In practice, this variability allows a service provider to estimate these rates and probabilities once with its combined data and then distribute them to the insurance companies without distributing the data. This is already good practice, for example, with the German Actuarial Association and its semi-Markov tables. \\
 We now introduce the different cash flows $B(t)$ that we use to model the cash flow of insurance contracts.
}\\
\noindent
Classical Markov modelling uses payment functions that allow payments to depend on the current state of the insured and on current jumps from one state to another. A classical cash flow that includes both of these option is a cash flow with a so-called one-dimensional canonical representation. 
\begin{definition}
	A stochastic process $B$ is said to have a \emph{one-dimensional canonical representation} if there exist real-valued functions $(B_{i})_{i}$ on $[0,\infty)$ with finite variation on finite intervals and measurable functions $(b_{ij})_{ij:i\neq j}$, which also have finite variation on finite intervals such that
	\begin{align}
		B(t)=\sum_{i \in \mathcal{Z}} \;\int\displaylimits_{(s,t]} I_i(u^-) B_i(\d u)+\sum_{i,j\in\mathcal{Z} \atop i\neq j}\;\int\displaylimits_{(s,t]} b_{ij}(u^-) N_{ij}(\d u),\quad t\geq s.\label{def:DefOfB}
	\end{align}
\end{definition}
In practice there often arise cash flows that do not have such a representation. Reasons for that can be contract modifications or general analysis of higher-dimensional cash flows. For this reason we allow for a more complex structure.
\begin{definition}
	A stochastic process $B$ is said to have a \emph{two-dimensional canonical representation} if there exist real-valued functions $(B_{i})_{i}$ on $[0,\infty)$ which are a difference of two non-decreasing upper continuous functions, real-valued functions $(B_{ij})_{ij}$ on $[0,\infty)^2$ which are a difference of two non-decreasing upper-continuous functions, and measurable and bounded real-valued functions $(b_{ikl})_{ikl}$, $(b_{ijkl})_{ijkl}$ on $[0,\infty)^2$ which are also the difference of two non-decreasing  upper-continuous functions, such that
	\begin{align}\begin{split}
			\boldsymbol B(\boldsymbol t)&=\sum_{i,j \in \mathcal{Z}}\; \;\int\displaylimits_{(s,\boldsymbol t]} I_i(u_1^-)I_j(u_2^-) B_{ij}(\d u_1, \d u_2)\\
			&\quad + \sum_{\substack{i,k,l \in\mathcal{Z} \\k\neq l}} \;\int\displaylimits_{(s,\boldsymbol t]} I_i(u_1^-)b_{ikl}(u_1^-,u_2^-) B_{i}(\d u_1)  N_{kl}(\d u_2)\\
			&\quad +\sum_{\substack{i,j,k,l \in \mathcal{Z}\\i \neq j, k\neq l}}\;\int\displaylimits_{(s,\boldsymbol t]} b_{ijkl}(u_1^-,u_2^-)  N_{ij}(\d u_1) N_{kl}(\d u_2),\quad \boldsymbol t\geq s.\label{def:DefOfB2}
		\end{split}
	\end{align}
\end{definition}
These processes with one- or two-dimensional canonical representations have been used by Bathke \& Christiansen \cite{bathke2022two} to model the conditional variance of future liabilities in an as-if Markov model and cash flows with special types of path-dependent payout functions. They do this by computing conditional expected values of stochastic processes with a one- or two-dimensional canonical representation.
We first define the conditional expected values of these cash flows\begin{align*}
	\boldsymbol A_z(t):=\mathbb{E}\Big[\int\displaylimits_{(s,t]^2}\boldsymbol B(\d \boldsymbol u)\Big|\xi=z\Big],\\
	 A_z(t):=\mathbb{E}\Big[\int\displaylimits_{(s,t]} B(\d u)\Big|\xi=z\Big].
\end{align*}
\textcolor{black}{In practice, a positive probability landmark should be used to condition the cash flows.}
These cash flows with a  two-dimensional canonical representation have already been studied by Bathke \& Christiansen in \cite{bathke2022two}. They showed
 \begin{align}\begin{split}
	\boldsymbol A_z(t)&=\sum_{i,j\in\mathcal{Z}} \;\int\displaylimits_{(s,t]^2}\boldsymbol P_{z,ij}(\boldsymbol u^-)B_{ij}(\d \boldsymbol u)\\
		&\quad+\sum_{\substack{i,k,l\in\mathcal{Z}\\k\neq l}}I_i(s) \int\displaylimits_{(s,t]^2}b_{ikl}(u_1^-,u_2^-) B_i(\d u_1) {P}_{z,k}(u_2^-) \Lambda_{z,kl}(\d u_2)\\
		&\quad+\sum_{\boldsymbol i\in\mathcal{Z}, \boldsymbol{j}\in\mathcal{Z}_{\neq}}\;\;\int\displaylimits_{(s,t]}\; \int\displaylimits_{(s,t] \times (s,u_1)} {b}_{j_2 i_1 j_1}(u_1^-,u_2^-)
		\boldsymbol{P}_{z,\boldsymbol i}(u_2^-,u_3^-) \boldsymbol \Lambda_{z,\boldsymbol{ij}}(\d u_2,\d u_3) B_{j_2}(\d u_1)\\
		&\quad+\sum_{\boldsymbol{i},\boldsymbol j\in\mathcal{Z}_{\neq}}\;\int\displaylimits_{(s,t]^2}b_{i_1j_1i_2j_2}(u_1^-,u_2^-) \boldsymbol P_{z,\boldsymbol i}(\boldsymbol u^-)\boldsymbol\Lambda_{z,\boldsymbol{ij}}(\d \boldsymbol u^-).\label{eq: expected value of two-dimensional cashflow}
		\end{split}
\end{align}where $\mathcal{Z}_{\neq}:=\{\boldsymbol i\in\mathcal{Z}|i_1\neq i_2\}$. In the one-dimensional case, Christiansen \cite{christiansen2021calculation} showed
\begin{align*}
	A_z(t)&=\sum_{i\in\mathcal{Z}}\int\displaylimits_{(s,t]}P_{z,i}(u^-)B_i(\d u)+\sum_{\substack{(i,j)\in\mathcal{Z}_{\neq}}}\int\displaylimits_{(s,t]} b_{ij}(u^-)P_{z,i}(u^-)\Lambda_{z,ij}(\d u).
\end{align*}

These $\boldsymbol A_z$ and $A_z$ can be understood as a functionals $\boldsymbol A_z({P}_{z},\boldsymbol P_{z},\Lambda_z,\boldsymbol \Lambda_z)$  and $ A_z({P}_{z},\Lambda_z)$ of the uni- and bivariate conditional transition rates and  transition probabilities. For the purposes of reserving in the insurer's balance sheet, we now need to estimate these conditional expected values. In practice, we first estimate the conditional transition rates and probabilities as seen in the previous section and then plug them into the functionals $\boldsymbol A_z$ and $A_z$. This works because of the following theorems.
\begin{theorem}\label{cor: Konvergenz der eindimensionalen aktuariellen schätzer}
For a stochastic process that has a one-dimensional canonical cash flow representation, the functional $ A_z$ is continuous in sup-norm in all arguments, meaning that the convergence in sup-norm in the domain results in convergence in the codomain. Additionally,
	\begin{align*}
		\lim_{\epsilon\rightarrow 0}\lim_{n\rightarrow \infty}\Big|A_z({P}_{z}^{(n,\epsilon)},\Lambda_z^{(n,\epsilon)})(t)-A_z({P}_{z}, \Lambda_z)(t)\Big|\stackrel{a.s.}{=}0,\quad\forall t>s.
	\end{align*}
\end{theorem}
The proof of this theorem follows at the end of this section. 
\begin{theorem}\label{continuity of V}
	For a stochastic process that has a two-dimensional canonical cash flow representation, the functional $\boldsymbol A_z$ is continuous in sup-norm in all $4$ arguments, meaning that the convergence in sup-norm in the domain results in convergence in the codomain.\\
	Additionally,
	\begin{align*}
		\lim_{\epsilon\rightarrow 0}\lim_{n\rightarrow \infty}\Big|\boldsymbol A_z({P}_{z}^{(n,\epsilon)},\boldsymbol P_{z}^{(n,\epsilon)},\Lambda_z^{(n,\epsilon)},\boldsymbol \Lambda_z^{(n,\epsilon)})(t)-\boldsymbol A_z({P}_{z},\boldsymbol P_{z},\Lambda_z,\boldsymbol \Lambda_z)(t)\Big|\stackrel{a.s.}{=}0,\quad\forall t>s.
	\end{align*}
\end{theorem}
The proof of this theorem follows at the end of this section. 
\begin{remark}
	All in all, the convergence we get is an almost sure uniform convergence, including the convergence in $\epsilon$. This type of convergence is different from the already established convergence for the one-dimensional cash flows by Christiansen \& Furrer \cite{christiansen2022extension}. They used $\mathcal{L}^1$ convergence in $p$ variation and manage to circumvent the use of the $\epsilon$ perturbation. The $\mathcal{L}^1$ convergence in $p$-variation is different from almost sure convergence in supremum norm, see Vitalis convergence theorem.
\end{remark}
The previous two theorems can now be used to estimate the conditional expected values of one-dimensional canonical cash flows, to estimate the conditional second moment of one-dimensional canonical cash flows that turn out to follow a stochastic process with a two-dimensional canonical representation, or to estimate the conditional expected values of path-dependent cash flows that can be rewritten as a stochastic process with a two-dimensional canonical representation. The general idea follows three steps.\begin{itemize}
	\item[i)] Estimate the conditional transition rates from data.
	\item[ii)] Plug in the estimated conditional transition rates and solve equation \eqref{eq: zwei-dimensionale Kolmogorov Gleichung} and equation \eqref{eq: GenKolmForwardEquation}.
	\item[iii)] Plug these estimators in the functionals $\boldsymbol A_z$ and $A_z$.
\end{itemize} 

\begin{proof}[Proof of Theorem \ref{continuity of V}]
	We show the theorem for $t=T$, for all $t<T$ the proof works the same. We start by showing the continuity of $\boldsymbol{A}_z$ with the help of Lemma \ref{lem:hadamard diff} and especially the fact that $B(F,G)\mapsto\int F(s)\d G(s)$ is Hadamard differentiable. For this differentiability to apply, we need bounded variation of all the integrators and integrands in equation \eqref{eq: expected value of two-dimensional cashflow}. 
	For the first integral, we have bounded variation in $n$ of $\boldsymbol P_{z}^{(n,\epsilon)}$ and $\boldsymbol P_{z}$ as we have seen in the proof of Theorem \ref{theo: convergence of Lambda and P} with the use of equation \eqref{I_represent_by_N}, additionally we have that $B_{ij}$ is of bounded variation because it was defined as the difference of two non-decreasing upper continuous functions.\\
	For the second summand, we need $
		\int\displaylimits_{(s,T]}b_{ikl}(u_1^-,u_2^-)B_i(\d u_1)
	 $ to be of bounded variation in $u_2$. For that we use the fact, that $b_{ikl}$ and $B_i$ are both functions that can be written as the difference of two non-decreasing functions $b^{+}_{ikl}$, $b^{-}_{ikl}$, $B_i^+$ and $B_i^-$. Plugging all of these in, we can bound the variation by \begin{align*}
		\Big\Vert \int\displaylimits_{(s,T]}b_{ikl}(u_1^-,u_2^-)B_i(\d u_1)\Big\Vert_{v_1}^{(1)}\leq& \int\displaylimits_{(s,T]}b_{ikl}^+(u_1^-,T)-b_{ikl}^+(u_1^-,s)B_i^+(\d u_1)\\
		&\;+\int\displaylimits_{(s,T]}b_{ikl}^-(u_1^-,T)-b_{ikl}^-(u_1^-,s)B_i^+(\d u_1)\\
		&\;+\int\displaylimits_{(s,T]}b_{ikl}^+(u_1^-,T)-b_{ikl}^+(u_1^-,s)B_i^-\d u_1)\\
		&\;+\int\displaylimits_{(s,T]}b_{ikl}^-(u_1^-,T)-b_{ikl}^-(u_1^-,s)B_i^-(\d u_1).
	\end{align*}
	Additionally, $P_z^{(n,\epsilon)}$ and $P_z$ are of uniform bounded variation, which can be seen similarly to $\boldsymbol P_z^{(n,\epsilon)}$ and $\boldsymbol P_z$ with equation \ref{I_represent_by_N}. For the last two summands, we use the same ideas and the fact, that $B_j$, $b_{ijk}$, and $b_{ijkl}$ are all the difference of two non-decreasing functions.
	Thus, we have proved the continuity of $A_z$ and with that found a way to estimate the conditional expected value of a stochastic process with a two-dimensional canonical cash flow representation.
\end{proof}
\begin{proof}[Sketch of the proof of Theorem \ref{cor: Konvergenz der eindimensionalen aktuariellen schätzer}]
The proof of this theorem uses the same ideas already used for the continuity of the two-dimensional cash flow.  We use the fact that $A_i$ and $a_{ij}$ both have finite variation and the fact, that integrals of these forms are Hadamard differentiable, see Lemma \ref{lem:hadamard diff}.
	\end{proof}

\section{Numerical Example}\label{Chap: numerical example}
\textcolor{black}{We consider a numerical example, where the bivariate conditional transition rates and transition probabilities arise naturally in the conditional expected values of a future cash flow. The focus of the numerical study is to illustrate how our two-dimensional cash flows from Section \ref{Chap: Section 5} can be applied in practice and how censoring affects the estimations. This is not intended to be a full statistical analysis, but merely an exploratory example.\\
The basic setup is similar to that used for the numerical study in the paper by Buchardt and M\o ller \cite{buchardt2015cash}. As a technical basis, we consider a survival model, i.e.\ a Markov model with the two states alive (1) and dead (4). The modelled insurance contract starts at the age of 40 and the retirement age is 65. The insured pays a lump sum of 100,000 at the beginning of the contract and an annual premium of 10,000 until the insured reaches the age of 65 and receives annual pension payments. All payments cease when the insured dies or reaches the age of 100. We get the following sojourn payment function: \begin{align*}
	B_1(\operatorname{dt} )&:=(-10,000\mathds{1}_{\{t<65\}}+p\mathds{1}_{\{100\geq t\geq 65\}})\operatorname{dt} ,
\end{align*} where $p$ is the annuity calculated on the technical basis according to the equivalence principle. The technical basis is generally used to price insurance contracts, while the extended model of the market basis is used for the insurers' balance sheets. The technical basis uses the transition intensity \begin{align*}
	\mu^\star_{14}(t):=0.005+10^{(5.728-10+0.038 t)},
\end{align*} and an interest rate of $0\%$.
The market basis consists of four states, with states (1) and (4) denoting alive and dead states, mirroring the technical basis. In addition, state (2) denotes the free policy option and state (3) denotes surrender, where free policy means that premiums are waived and benefits are reduced accordingly. Surrender means a complete cessation of both benefits and payments with a final transitional payment, usually in the form of the current reserve. We therefore introduce two options that model policyholder behaviour for the market basis. The additional payment functions are modelled as follows:
\begin{align*}
	b_{13}(t)&=V_1(t),\\
	b_{23}(t,u)&=V_1^{+}(t)\rho(u),
\end{align*} for the transition payments, where $V_1^{+}(t)$ is the technical reserve of the pension payments, and 
\begin{align*}
B_2(\operatorname{dt},u)&:=\rho(u)p\mathds{1}_{\{100\geq t\geq 65\}}\operatorname{dt} ,
\end{align*} for the sojourn payments, where $p$ is the premium that is paid in state (1). The factor $\rho(u)$ defines the reduction of the benefits in the free policy state (2). It is calculated by 
\begin{align*}
	\rho(u)=\frac{A_1^{+}(u)-A_1^{-}(u)}{A_1^+(u)},
\end{align*} where $A_1^{-}(t)$ is the technical reserve of premium payments. The dependence of the payments functions on $u$ models the dependence on the transition time to state (2). For this reason, bivariate transition rates are natural in this application. See the work of Ahmad et al.\ in \cite{Ahmad_Buchardt_Furrer_2022} for the theoretical background of free policy option calculations.\\
We model the market basis as a semi-Markov model with transition intensities \begin{align*}
	\mu_{14}(t)&:=\mu^\star_{14}(t),\\
	\mu_{24}(t)&:=\mu^\star_{14}(t),\\
	\mu_{12}(t)&:=0.1\mathds{1}_{\{t\leq 65\}}\\
	\mu_{13}(t)&:=0.05\mathds{1}_{\{t\leq 65\}}\\
	\mu_{23}(t,u)&:=(0.05+0.2\mathds{1}_{\{u\in[\frac{1}{2},\frac{5}{2})\}})\mathds{1}_{\{t\leq 65\}},
\end{align*}and an interest rate of $3\%$ annually. We estimate the conditional expected future cash flow of the market basis in an as-if Markov model, i.e.\ the landmark is $\xi=Z(40)$ and since 40 is the starting point of the insurance contract, we assume that each insured is in state (1) at $s=40$. For an as-if Markov model and a cash flow with scaled payments such as the free policy option, Bathke and Christiansen in \cite{bathke2022two} found a two-dimensional canonical representation for the resulting conditional expected future cash flow. Adapting it to this example, we get \begin{align}
\begin{split}
\boldsymbol A_1(t)&=\int\displaylimits_{(40,t]} \frac{\kappa(40)}{\kappa(u)} p_{1,1}(u^-)  B_1(\d u)  + \int\displaylimits_{(40,t]}  \frac{\kappa(40)}{\kappa(u)} b_{13}(u)  p_{1,1}(u^-)   \Lambda_{1,(1,3)}(\d u)\\  &\quad+\int\displaylimits_{(40,t]}\;\int\displaylimits_{(40,65]\times (40,t)}\frac{\kappa(40)}{\kappa(t)}\rho(u_1)\bigg( 
\boldsymbol p_{1,(1,1)}(\boldsymbol{u^-}) \boldsymbol\Lambda_{1,(1,1)(2,2)}(\mathrm{d}\boldsymbol{u})  \\
&\hspace{50mm}-
\sum_{k \in \{3,4\}} \boldsymbol p_{1,(1,2)}(\boldsymbol{u}^-) \boldsymbol\Lambda_{1,(1,2)(2,k)}(\mathrm{d}\boldsymbol{u})
\bigg)p\mathds{1}_{\{t\geq 65\}}\mathrm{d} t\\
&\quad+ \int\displaylimits_{(40,t]\times(40,t]}\frac{\kappa(s)}{\kappa(u_1)}\rho(u_1)b_{23}(u_2,u_1)\boldsymbol p_{1,(1,2)}(\boldsymbol{u^-})\boldsymbol\Lambda_{1,(1,2)(2,3)}(\mathrm{d}\boldsymbol{u}),\label{eq: twodimensional cash flow}
\end{split}
\end{align}where $t\leq 100$ and $\kappa(\cdot)$ is the value of a deterministic cash account. The conditional expected value of this cash flow is an example where the estimation using the as-if Markov approach is natural, because we have jumps from state (2) to state (3) or (4) in the cash flow that are dependent on the jump from state (1) to state (2). Such dependencies are the strength of the conditional bivariate estimation.\\
We introduce some technical background to the estimation of this expected future cash flow in the as-if Markov model.
The estimation of the conditional bivariate transition rates and transition probabilities for the market basis can be simplified. We can reduce the required transition rates to $\Lambda_{(1,2),(2,4)}$,$\Lambda_{(1,2),(2,3)}$, $\Lambda_{(1,2),(1,2)}$, $\Lambda_{(1,4),(1,4)}$ and $\Lambda_{(1,3),(1,3)}$ where the last two are only needed to estimate the transition probability $p_{11}$. This reduction originates from the reduced number of possible jumps in the market basis. The required transition rates can also be reduced. We only need $p_{11},p_{12}$ and $p_{13}$. This greatly reduces computation time and space.\\
Counting the number of steps required to compute the transition rates and probabilities, we get a time complexity of order $\mathcal{O}(n^2)$ and a space complexity of order $\mathcal{O}(n^2)$. Intuitively, we get this complexity because there are $\mathcal{O}(n^2)$ time points on the time grid given by the data, and we need to recursively compute the estimator on this grid.\\
To compute the plug-in estimator, we can use a type of differential equation, in this case a simple recursive formula due to the finiteness of the grid points. This reduces the computational complexity for the computation of the conditional expected future cash flow from $\mathcal{O}(n^3)$ to $\mathcal{O}(n^2)$. We use the decomposition \begin{align*}
	&\boldsymbol A_1({P}_{1}^{(n,\epsilon)},\boldsymbol P_{1}^{(n,\epsilon)},\Lambda_1^{(n,\epsilon)},\boldsymbol \Lambda_1^{(n,\epsilon)})(t) - \boldsymbol A_1(0)
	=:C_1(t)+C_2(t)+C_3(t),\quad t\geq 40,
\end{align*}where $C_1$ contains the first two, $C_2$ the third and $C_3$ the fourth summand of the equation \eqref{eq: twodimensional cash flow}, with the estimated transition rates and transition probabilities plugged in. Now all the three summands can be calculated recursively. Let $\{t_n\}_{n\in I}$ be the the set of all jump points defining the time grid imposed by the data. We start with $t_0=40$ and $C_1(40)=C_2(40)=C_3(40)=0$. We get \begin{align*}
	C_1(t_n)=C_1(t_{n-1})+\int_{t_{n-1}}^{t_n}p^{(n,\epsilon)}_{1,j}(s-)\Big(B_1(\mathrm{d} s)+ b_{13}(s)\Lambda^{(n,\epsilon)}_{1,13}(\mathrm{d} s)\Big).
\end{align*}For the second one, we get \begin{align*}
	C_2(t_n)&=C_2(t_{n-1})+\int_{(40,t_{n-1}-)}^{(65,t_n-)}\int_{u_2}^{t_n}B_2(\mathrm{d} s,u_1)\bigg( 
	\boldsymbol P^{(n,\epsilon)}_{1,(1,1)}(\boldsymbol{u-}) \boldsymbol\Lambda^{(n,\epsilon)}_{1,(1,1)(2,2)}(\mathrm{d}\boldsymbol{u})  \\
	&\hspace{62mm}-
	\sum_{k \in \{3,4\}} \boldsymbol P^{(n,\epsilon)}_{1,(1,2)}(\boldsymbol{u}-) \boldsymbol\Lambda^{(n,\epsilon)}_{1,(1,2)(2,k)}(\mathrm{d}\boldsymbol{u})
	\bigg)\\
	&\hspace{18.5mm}+\int_{t_{n-1}}^{t_n}p\mathds{1}_{\{t\geq 65\}}\mathrm{d} (s)\int_{(40,40)}^{(65,t_{n-1}-)}\rho(u_1)\bigg( 
	\boldsymbol P^{(n,\epsilon)}_{1,(1,1)}(\boldsymbol{u-}) \boldsymbol\Lambda^{(n,\epsilon)}_{1,(1,1)(2,2)}(\mathrm{d}\boldsymbol{u})  \\
	&\hspace{62mm}-
	\sum_{k \in \{3,4\}} \boldsymbol P^{(n,\epsilon)}_{1,\boldsymbol{j}}(\boldsymbol{u}-) \boldsymbol\Lambda^{(n,\epsilon)}_{1,(1,2)(2,k)}(\mathrm{d}\boldsymbol{u})\bigg).
\end{align*} The last integral in the formula for $C_2$ can again be calculated recursively in every step. For $C_3$, we get \begin{align*}
	C_3(t_n)=C_3(t_{n-1})&+ \int_{(40,t_{n-1})}^{(65,t_n)}
	b_{23}(u_2,u_1)\boldsymbol P^{(n,\epsilon)}_{1,(12)}(\boldsymbol{u-})\boldsymbol\Lambda^{(n,\epsilon)}_{1,(1,2)(2,3)}(\mathrm{d}\boldsymbol{u}).
\end{align*}
We estimate in R \cite{R_citation} by generating artificial data following the semi-Markov model described as the market basis using the package \emph{AalenJohansen} by Bladt and Furrer \cite{Bladt_Furrer_R_package}.\\
 We simulate $5,000$ independent and identically distributed realisations for the defined insurance contract with scaled payments.  We use $\operatorname{Unif}(65,120)$ distributed random variables as right censoring.}
 \begin{figure}[h!]
 	\centering 	\subfloat[$n=1000$]{\includegraphics[width=\linewidth]{"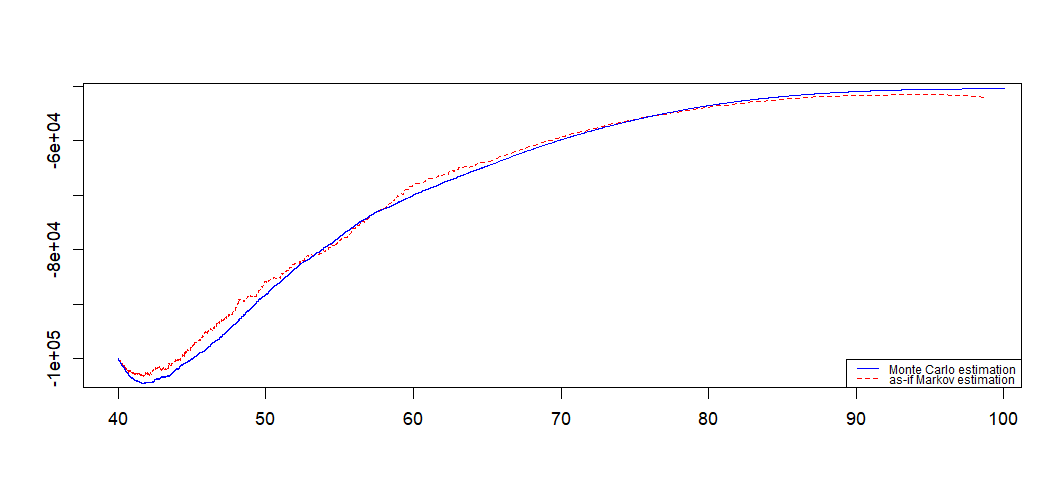"} 	\label{fig: plot_reserve_1000}}
 	\hfill
 	\subfloat[$n=5000$]{\includegraphics[width=\linewidth]{"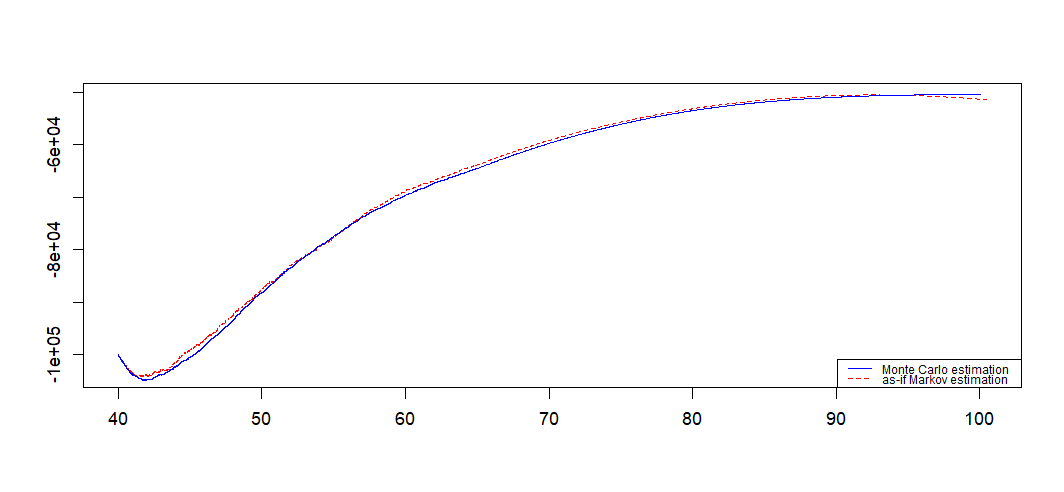"} 	\label{fig: plot_reserve_5000}}
\caption{Expected future cash flow $\boldsymbol{A}_1(t)$ with time horizon $T=100$ in the as-if Markov model estimated with $n=1000$ and $n=5000$ censored individuals, the plug-in estimator and the landmark Nelson-Aalen and Aalen-Johansen estimators versus Monte Carlo estimation with $n=10000$ uncensored individuals.  }
\end{figure}

\textcolor{black}{
  To illustrate the numerical feasibility and strength of the plug-in estimator, we plot in figures \ref{fig: plot_reserve_1000} and \ref{fig: plot_reserve_5000} the estimate for the conditional expected future cash flow of the insurance contract in an as-if Markov model on the market basis for $n=1000$ and $n=5000$ individuals, together with a Monte Carlo simulation of the same conditional expected future cash flow without censoring and $10,000$ individuals for comparison. We see that the bivariate estimator, with $n=5000$, estimates the conditional expected future cash flow of the particular insurance product  similarly to the Monte Carlo estimator, which is using uncensored data and more individuals. Thus, the estimator appears to be a good fit for this type of estimation problem.\\
   A challenge for the interpretation of the estimator $P^{(n,\epsilon)}_{z,(jk)}$ is that in the censored case the bivariate estimator may not satisfy the expected monotonicity and positivity properties. This property is very similar to the findings in \cite{Dabrowska1988} where they found the Kaplan-Meier estimator on the plane to be non-monotonic. This explains the higher volatility observed in the plug-in estimator for $n=1000$ but does not change the fact that the estimator consistently estimates the conditional expected future cash flow, even when censoring is involved. Moreover, the estimators are easy to apply, since we do not need to solve any continuous differential equations, but were able to obtain recursive formulas for the estimators at each level. \\
 We conclude that, despite numerical challenges, the bivariate estimator can be successfully applied to the estimation of conditional expected values of future cash flows in non-Markov models.
The bivariate landmark Nelson-Aalen and Aalen-Johansen estimators, combined with the plug-in estimators for conditional expected future cash flows, provide a versatile tool because the theory applies not only to insurance contracts with contract modifications, but also to other cash flows that have a two-dimensional representation. }

	\section{Concluding Remarks}\label{SectionConclusion}
In this paper, we generalized the landmark Nelson-Aalen and the landmark Aalen-Johansen estimators to the estimation of bivariate conditional transition rates and bivariate transition probabilities in general non-Markov models as introduced by Bathke \& Christiansen \cite{bathke2022two}.
The analysis of intertemporal dependencies via bivariate conditional transition rates is necessary in non-Markov multi-state modelling, in contrast to classical Markov modelling, because the absence of the Markov assumption leads to incomplete distribution knowledge gained from univariate conditional transition rates. \\ 
	The use of analytical techniques well-known from bivariate Kaplan-Meier estimation and conditional Aalen-Johansen estimation leads to uniform convergence of both estimators. We obtain a blueprint for the estimation of bivariate conditional transition rates and transition probabilities from right censored data.\\
	Future research endeavours could further address the question of asymptotic normality and bootstrapping. Additionally, an extension of the uniform convergence to the entire plane, possibly via martingale techniques, is of interest.\\
	Finally, this study closes a gap between techniques in survival analysis and premium and reserve calculation for insurance contracts. In the case of an insurance contract with scaled payments, we construct estimators for the first moment of all future payments \textcolor{black}{and present a proof of concept with a numerical example using simulated data}. This is the basis of any type of risk assessment, where the state process of the insured follows the non-restrictive non-Markov model. Furthermore, estimating bivariate conditional transition rates and transition probabilities allows the estimation of second moments of future liabilities with a canonical one-dimensional cash flow in the non-Markov model. Altogether, this can be used to evaluate the widespread use of Markov assumptions in the reserving and premium calculation for insurance portfolios.
	\subsection*{Declarations}
	\textbf{Conflict of interest} The author has declared no conflict of interest.
	\subsection*{Acknowledgements}
The author would like to express gratitude to Marcus Christiansen and Christian Furrer for their invaluable discussions regarding the direction and theoretical background of this paper.

	\bibliographystyle{acm}
	\bibliography{mybib.bib}
		\appendix
	\section{Appendix}
	The following first two lemmas are two-dimensional generalisations of well-known statements in the field of product integrals. These generalisations have already been mentioned in Gill \& Johansen \cite{gill1990survey} and Gill, van der Laan \& Wellner \cite{gill1995inefficient}.
	\begin{lemma}\label{lem Peano as solution}
		For a monotonically increasing real-matrix valued function $\Lambda$ and a function $\phi(\boldsymbol t)$, the equation \begin{align}
			Y(\boldsymbol t)=\phi(\boldsymbol t)+\int\displaylimits_{( s,\boldsymbol t]}Y(\boldsymbol u^-)\Lambda(\d \boldsymbol u),\label{eq: volterra equation 2}
		\end{align} has the solution
		\begin{align}
			\begin{split}
				\tilde P(\boldsymbol t)&=\phi(\boldsymbol t)+\int\displaylimits_{(s,\boldsymbol t]}\phi(\boldsymbol u^-)\Lambda(\d \boldsymbol u)\mathcal{P}((\boldsymbol u,\boldsymbol t],\Lambda),
			\end{split}\label{eq:Peano-expansion}
		\end{align} where we use the Peano-series definition \begin{align*}
			\mathcal{P}((s,\boldsymbol t],\Lambda)&:=Id+\sum\displaylimits_{n=1}^\infty\;\; \idotsint\displaylimits_{s<\boldsymbol u^{(1)}<...<\boldsymbol u^{(n)}\leq \boldsymbol t}\;\;\Lambda(\d \boldsymbol u^{(1)})\cdots\Lambda(\d \boldsymbol u^{(n)}).
		\end{align*}
		In the definition of the Peano-series we use the usual partial ordering of $\mathbb{R}^2$ for the definition of the domain. 
	\end{lemma}
	\begin{proof}
		See the proof of Proposition 3.1 from Gill, van der Laan \& Wellner \cite{gill1995inefficient}.
	\end{proof}

	The next Lemma is a generalisation of the well-known Duhamel equality. 
	
	\begin{lemma}\label{lem: Duhamel}
		For monotonically increasing functions $A,B:\mathbb{R}^2\rightarrow \mathbb{R}^{n\times n}$ we have the following equation\begin{align*}
			\mathcal{P}(R,A)-\mathcal{P}(R,B)=\int\limits_{ R} \mathcal{P}(R\cap (\boldsymbol 0,s),A)\left(A(\d s)-B(\d s)\right)\mathcal{P}(R\cap(s,\infty),B),
		\end{align*}for a rectangle $R\subset\mathbb{R}^2$.
	\end{lemma}
	\begin{proof}
		See the proof of Proposition 3.3 from Gill, van der Laan \& Wellner \cite{gill1995inefficient}.
	\end{proof}
\begin{lemma}\label{lem:hadamard diff}
	The functional \begin{align*}
A:(F,G)\mapsto\int \frac{1}{F(s)}\d G(s)		
	\end{align*}is continuous at any point $(F,G)$ on the space of cadlag functions with bounded variation endowed with the supremum norm, where $\frac{1}{F}$ is of bounded variation and bounded away from $0$ and $G$ is of bounded variations for sequences $\frac{1}{F_n}$ and $G_n$ that are uniformly of bounded variation as well.
 
\end{lemma}
\begin{proof}
	This follows immediately by the weak continuous Hadamard differentiability of the functional $B:(F,G)\mapsto \int F(s)\d G(s)$ and the fact that $F\mapsto\frac{1}{F}$ is also weak continuous Hadamard differentiable as long as $F\geq\delta>0$. The statement can be found in Gill, van der Laan \& Wellner\cite{gill1995inefficient} Lemma 5.1. 
\end{proof}

	Additionally, we need a two-dimensional generalisation of the well known Campbell theorem, see for example Milbrodt \& Helbig \cite{milbrodt1999mathematische}.
	\begin{theorem}[Campbell Theorem]
		Let $\eta$ be a point process on $(\mathbb{R}^d,\mathbb{B}(\mathbb{R}^d))$ with intensity measure $\lambda$ and let $u:\mathbb{R}^d\rightarrow\mathbb{R}$ be a measurable function. Then \begin{align*}
			\int u(x)\eta(\d x)
		\end{align*} is a random variable and \begin{align}
			\mathbb{E}\left[\int u(x)\eta(\d x)\right]=\int u(x)\lambda(\d x),\label{eq:campbell}
		\end{align}
		whenever $u\geq0$ or $\int | u(x)| \lambda(\d x)<\infty$.
	\end{theorem}
	\begin{proof}
The proof of this can be easily adapted from the proof of the one-dimensional Campbell Theorem, see for example Milbrodt \cite{milbrodt1999mathematische}.
	\end{proof}
	For the estimation of integrals, we use the so called one and two-dimensional variation of a process. We do a short introduction here. 
	
	\begin{definition}\label{def:variation}
		We define the one-dimensional variation of a function $f:[a,b]\rightarrow \mathbb{R}$ as \begin{align*}
			\Vert f\Vert_{v_1}^{(1)}:= \Vert f\Vert_\infty+ \sup_{\tau=\{t_i\}_{i\leq n}}\sum\displaylimits_{i=1}^{n}|f(t_i)-f(t_{i-1})|,
		\end{align*} where the supremum goes over all finite partitions $\tau$ of $[a,b]$. The two-dimensional variation of a function $f:[a,b]\times[c,d]\rightarrow \mathbb{R}$ is defined as \begin{align*}
			\Vert f\Vert_{v_1}^{(2)}&:=\sup_{\tau=\{t_i\}\times\{\tilde t_j\}_{i,j\leq n}}\sum_{i=1}^n\sum_{j=1}^n |f(t_{i-1},\tilde t_{j-1})-f(t_{i-1},\tilde t_{j})-f(t_{i},\tilde t_{j-1})+f(t_{i},\tilde t_{j})|\\
			&\;+ \Vert f(a,\cdot)\Vert_{v_1}^{(1)}+ \Vert f(\cdot,c)\Vert_{v_1}^{(1)}-\Vert f\Vert_\infty,
		\end{align*}where the supremum goes over all finite partitions of $[a,b]\times[c,d]$.
	\end{definition} This variation definition is mainly used for the following integral inequality 
\begin{lemma}\begin{align*}
		\left|\int f(\boldsymbol u) g(\d \boldsymbol u)\right|\leq \Vert f\Vert_\infty \Vert g\Vert_{v_1}^{(2)},
	\end{align*}when $f$ is bounded and $g$ has finite two-dimensional variation.
\end{lemma}
	The proof bases on work from Adams \& Clarkson \cite{clarkson1933definitions} and \cite{adams1934properties} , see Theorem 5.
\end{document}